\documentclass[12pt,a4paper]{amsart}
\usepackage[utf8]{inputenc}
\usepackage{amsmath}
\usepackage{amsfonts}
\usepackage{amssymb}
\usepackage{amsthm}
\usepackage{mathrsfs}
\usepackage{color}
\usepackage{a4wide}
\usepackage{hyperref}
\usepackage{comment}
\usepackage{tikz-cd}
\usepackage{tikz}
\usepackage{float}
\usepackage{caption}

\newcommand{\supp}{\mbox{$\mathbf{s}$}}

\theoremstyle{plain}
\newtheorem{Thm}{Theorem}[section]
\newtheorem{Cor}[Thm]{Corollary}
\newtheorem{Lem}[Thm]{Lemma}
\newtheorem{Prop}[Thm]{Proposition}

\newtheorem{corollary}[Thm]{Corollary}
\newtheorem{lemma}[Thm]{Lemma}

\theoremstyle{definition}		
\newtheorem{defn}[Thm]{Definition}
\newtheorem{Rem}[Thm]{Remark}

\begin{document}

\title[{Finitary conditions for graph products of monoids}]%
{{Finitary conditions for graph products of monoids}}

\author{Dandan Yang}
\address{School of Mathematics and Statistics, Xidian University, Xi'an 710071, P. R. China }
\email{ddyang@xidian.edu.cn}
\author{Victoria Gould}

\address{Department of Mathematics, University of York, York YO10 5DD, UK}
\email{victoria.gould@york.ac.uk}

\keywords{Monoid,   weakly left noetherian, left ideal Howson, finitely left equated}

\subjclass[2020]{20M05, 20M10, 20M30}

\thanks{This work was  supported by the Engineering and Physical Sciences Research Council Grant EP/V002953/1, the National Natural Science Foundation of China Grant 12171380,  the Natural Science Basic Research Program of Shaanxi Province Grant 2023-JC-JQ-04,  the Shaanxi Fundamental Science Research Project for Mathematics and Physics Grant  22JSQ034, the Fundamental Research Funds for the Central Universities Grant QTZX25093, and Xidian University Specially Funded Project for Interdisciplinary Exploration Grant TZJH2024007.} 

\begin{abstract} Graph products of monoids provide a common framework for free products and direct products.  This paper   investigates the interaction of certain finitary conditions with the graph product construction. Specifically, we examine the conditions of being weakly left  noetherian (that is, every left ideal is finitely generated) and weakly left coherent (that is, every finitely generated left ideal has a finite presentation)
and the related conditions of the  ascending chain condition on principal left ideals, being left ideal Howson, and being finitely left equated.  All of these conditions, and others, are preserved under retract; as a consequence, if a graph product has such a property, then so do all the constituent  monoids. We show that the converse is also true for  all the conditions listed except that of being weakly left noetherian. In the latter case we precisely determine those graph products of monoids which are weakly left noetherian. 
\end{abstract}

\maketitle

\section{Introduction}
A finitary condition for a class of algebras  is one possessed by all its finite members: this paper examines finitary conditions in the context of  graph products of monoids. Let $\Gamma=(V,E)$ be a simple graph and let $\mathcal{M}=\{ M_\alpha: \alpha\in V\}$ be a set of monoids, referred to as vertex monoids; we do not assume that $\Gamma$ is finite. The graph product of $\mathcal{M}$ with respect to $\Gamma$ is the freest monoid generated by  $\bigcup_{\alpha\in V}M_{\alpha}$ such that each $M_{\alpha}$ is a submonoid, and elements in monoids
$M_{\alpha}$ and $M_\beta$ commute if $(\alpha,\beta)\in E$. A formal definition in terms of presentations is given in Section~\ref{sec:prelim}.  Free products and restricted direct products (that is, the submonoid of elements of direct products with finite support) are instances of graph products at opposite extremes.  In the former $E=\emptyset$, so that non-identity elements in distinct vertex monoids do not commute. In the latter $E=V\times V$, so that all elements in distinct vertex monoids commute. There are many sources and motivations for  graph products; in the form given here the concept was introduced by Green \cite{Green:1990} for groups and by da Costa \cite{dacosta:2001} for monoids. Trace monoids  and right angled Artin groups, where $\Gamma$ is finite, and the vertex monoids are 
finitely generated free monoids or free groups, respectively, provide important instances.  The former were introduced in the study of combinatorial problems for rearrangements of words; they have been extensively studied by mathematicians and computer scientists (see \cite{diekert:1990,duncan:2020}).

It is a  natural question  whether certain monoid properties pass between a graph product and its vertex monoids. Algorithmic properties have been extensively considered \cite{dacosta:2001,Green:1990,hermiller:1995} as have algebraic properties \cite{barkauskas:2007,fountain:2009,gould:2023}.  Recently, the authors, together with Cho and Ru\v{s}kuc, have shown that a graph product of monoids is residually finite if and only if so is each vertex monoid \cite{cho:2024}. 
In this article we again address  finitary conditions, in this case those related to noetherianity and coherency. 

Recall that a ring is left noetherian if every left ideal is finitely generated and is left coherent if every finitely generated left ideal has a finite presentation. For monoids there are two natural choices for replacements for each of these notions. Those depending on  left congruences are known as  left noetherianity and left coherency and those depending on left ideals  as weak  left  noetherianity and  weak left   coherency. In terms of understanding their behaviour with respect to direct products, let alone graph products, left coherency and left noetherianity  are mysterious. The direct product of two left coherent monoids need not be left coherent, although the direct product of a left coherent monoid with a finite monoid is  \cite{gould:2020}. It is not known whether an arbitrary  direct product of left noetherian monoids is left noetherian, although again it is known that the direct product of a left noetherian monoid with a finite monoid is  again left noetherian \cite{miller:2019}. 

We are therefore led  to examine weak left noetherianity  and weak left coherency. The former property again does not pass up  perfectly to a graph product from the vertex  monoids.   For the graph product to be weakly left noetherian, in addition to each  vertex monoid being weakly left noetherian  all but finitely many   must be groups.  Further,  the  non-group  vertex monoids must be adjacent to all but finitely many vertex monoids; the precise conditions are   stated in  Theorem~\ref{wln}. Via a standard argument, a monoid is left noetherian if and only if it has the ascending chain condition on left ideals. If we consider the ascending chain condition on {\em principal} left ideals (ACCPL), then we do have a completely positive result - Theorem \ref{acc} shows that a graph product has  ACCPL   if and only if so does each vertex monoid. The situation for weak left coherency is equally pleasing. Our main result, Theorem~\ref{thm:wlc}, gives that a graph product of monoids is weakly left coherent if and only if so is each vertex monoid. This we do by examining two finitary conditions of importance in their own right: that of being left ideal Howson and that of being
finitely left equated (FLE). The former is also called being finitely (left) aligned \cite{exel:2018} and the latter is equivalent to  monogenic left monoid acts having finite presentations \cite{dasar:2024};  their conjunction is equivalent to being weakly left coherent \cite{gould:1992}. In Theorems~\ref{thm:howson} and \ref {thm:aagh} respectively we show that  a graph product of monoids is left ideal Howson (respectively, FLE) if and only if so is each vertex monoid. Consequently, we deduce Theorem~\ref{thm:wlc}, which says that a graph product of monoids is weakly left coherent  if and only if so is each vertex monoid.

 Figure~\ref{fig:theprops}  indicates the relationships between the finitary conditions considered. Note that each condition has a left-right dual for which the corresponding left-right dual result  holds. For convenience we omit the reference to `left' below.

\begin{figure}[H] 
\centering

\begin{tikzpicture}[scale=1.1]
\draw (0,0)--(0,4);\draw (0,4)--(6,4);\draw (6,4)--(6,0);\draw (6,0)--(0,0);\draw (0,3)--(6,3);\draw (0,2)--(6,2);\draw (1.5,0)--(1.5,2);\draw (3,1)--(3,3);\draw (3,1)--(6,1);\draw (4.5,0)--(4.5,1);

\coordinate(A_1) at(3.7,3.5);\coordinate(A_2) at(3,2.5);\coordinate(A_3) at(5.2,2.5);\coordinate(A_4) at(1.3,1);\coordinate(A_6) at(4.3,0.5);\coordinate(A_7) at(5.9,1.5);\coordinate(A_8) at(5.7,0.5);\coordinate(A_9) at(6.5,2);

\node at(A_1) [left]{\small{noetherian}};\node at(A_2) [left]{\small{weakly ~noetherian}};\node at(A_3) [left]{\small{coherent}};\node at(A_4) [left]{\small{ACCP}};\node at(A_6) [left]{\small{ideal~Howson}};\node at(A_7) [left]{\small{weakly ~coherent}};\node at(A_8) [left]{\small{FE}};\node at(A_9) [right]{\small{weakly coherent $=$ ideal~Howson $+$ FE}};
\end{tikzpicture}
\captionsetup{margin={-6.3cm,0cm}}
\caption{Implications flow downwards}         
\label{fig:theprops}
\end{figure}
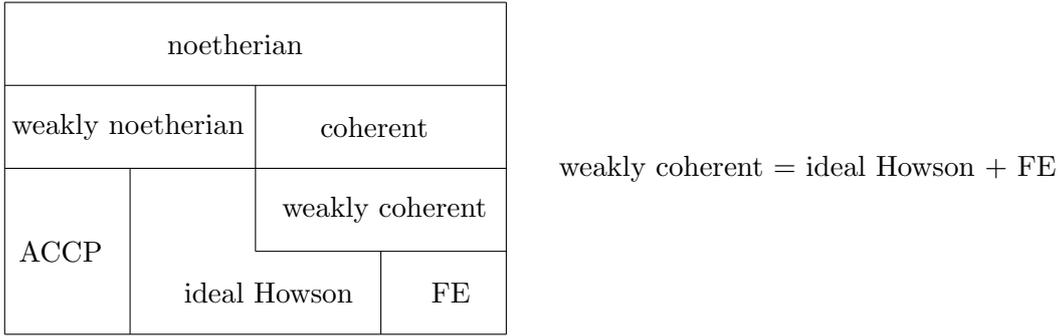 

All the properties in Figure~\ref{fig:theprops} are preserved under retract and consequently, if a graph product of vertex monoids has any of these conditions, then so do all the vertex monoids. We refer to \cite{gould:2020} for this fact for left coherency and \cite{miller:2019} for left noetherianity. We give references or brief proofs for the other conditions as we meet them later in the paper.

The structure of the paper is as follows. We give the necessary preliminaries for graph products in Section~\ref{sec:prelim}. The precise formulations of each of the finitary conditions we consider will appear at the start of the relevant sections. In Section~\ref{sec:AAPL} we consider
ACCPL and prove it is preserved by taking graph products. Section~\ref{sec:wln} focuses on determining the situation in which the graph product of weakly left noetherian monoids is weakly left noetherian; exceptionally for this article, this depends on the nature of the graph, and not just the vertex monoids. For the purposes of the more intricate arguments surrounding weak left coherency, we present some technical tools in Section~\ref{sec:red}.
We then show in Sections \ref{sec:leftidealH} and \ref{sec:L} that being left ideal Howson or FLE is always preserved under the operation of taking graph product. Finally, we put these together in  Section~\ref{sec:wlc} to give the desired result for weak left coherency.

\section{Preliminaries}\label{sec:prelim}

We outline the notions and basic results required to read this article.  For further details we recommend \cite{howie:1995} for semigroups, and \cite{dacosta:2001, gould:2023, Green:1990}  for graph products of groups or  monoids.

\subsection{Monoid presentations}\label{sub:pres}  Let $X$ be a set. The  {\em free monoid} $X^*$ on $X$ consists of all words over $X$ with operation of juxtaposition.
We denote a  non-empty word by $x_1\circ \dots \circ x_n$ where $x_i\in X$ for $1\leq i\leq n$; here $n$ is the {\em length} of the word. We  also
 use $\circ$ for juxtaposition of words. The empty word is denoted by $\epsilon$ and is the identity of $X^*$. Throughout, our convention is that if we say $x_1\circ\dots \circ x_n\in X^*$, then we mean $x_i\in X$ for all $1\leq i\leq n$, unless we explicitly say otherwise. A {\it subword} of $x=x_1 \circ \cdots \circ x_n$ is defined as a word with a form $y=x_{i_1}\circ \cdots \circ x_{i_m}$ where $1\leq i_1<i_2<\cdots<i_m\leq n$; we do not assume that $y$ is a factor of $x$ in $X^*$.

A {\em monoid presentation} $\langle X\mid R\rangle$, where $X$ is a set and $R\subseteq  X^*\times X^*$, determines the monoid $X^*/R^{\sharp}$, where $R^{\sharp}$ is the congruence on $X^*$ generated by $R$. In the usual way, we identify $(u,v)\in R$ with the formal equality $u=v$ in a presentation $\langle X\mid R\rangle$. We denote elements of the quotient by $[w]$, where $w\in X^*$.

\subsection{Graph products of monoids}  \label{sub:gp}

In this article, all graphs are simple, that is, they are undirected with no multiple edges or loops.
   Let $\Gamma=(V,E)$ be a  graph;
here $V$ is the  non-empty set of {\em vertices} and  $E$ is  the set of {\em edges} of  $ \Gamma$. We denote an edge joining vertices $\alpha$ and $\beta$ by
$(\alpha,\beta)$ or (since our graph is undirected) by $(\beta,\alpha)$.

 \begin{defn}\label{defn:graphprodmonoids} \cite{dacosta:2001, Green:1990}  Let
$\Gamma=(V,E)$ be  graph and let $\mathcal{M}=\{M_\alpha: \alpha\in V\}$ be a set of mutually disjoint monoids.  For each $\alpha\in V$, we write $1_{\alpha}$ for the identity of $M_{\alpha}$ and put $I=\{ 1_{\alpha}: \alpha\in V\}$.
The {\em graph product} $\mathscr{GP}=\mathscr{GP}(\Gamma,\mathcal{M})$ of
 $\mathcal{M}$   with respect to $\Gamma$ is defined by the monoid presentation
 \[\mathscr{GP}=\langle X\mid R  \rangle\]
 where $X=\bigcup_{\alpha\in V}M_\alpha$ and the relations in $R= R_{\textsf{id}}\cup R_{\textsf{v}}\cup R_{\textsf{e}}$ are given by:
\[\begin{array}{rcl}R_{\textsf{id}}&=&\{   1_\alpha=\epsilon:  \alpha\in V\},\\

R_{\textsf{v}}&=&\{   x\circ y=xy:  \ x,y\in M_\alpha,\alpha\in V\},\\

R_{\textsf{e}}&=& \{ x\circ y=y\circ x: x\in M_\alpha,\, y\in M_\beta, (\alpha,\beta)\in E)\}.\end{array}\]
\end{defn}

The monoids $M_{\alpha}$ in Definition~\ref{defn:graphprodmonoids} are  referred to as \textit{vertex monoids}. We may also say for brevity that $\mathscr{GP}(\Gamma,\mathcal{M})$ is the {\em  graph product of the monoids $M_{\alpha},\alpha\in V$}. We use the notation $\equiv$ for $R^{\sharp}$ and say that
$u$ is {\em equivalent} to $v$ if $[u]=[v]$, that is, if $u\equiv v$. We may use $1$ to denote the identity of $\mathscr{GP}$, that is, the equivalence class
$[\epsilon]$.

\begin{Rem}\label{rem:freeprod} Let $\mathcal{M}=\{M_\alpha: \alpha\in V\}$ be a set of mutually disjoint monoids.
\begin{enumerate}
\item[$\bullet$]  If $E=\emptyset$ then the  graph product $\mathscr{GP}(\Gamma,\mathcal{M})$  is isomorphic to the  free product of the monoids in $\mathcal{M}$.

\item[$\bullet$] If $E=V\times V$ then the  graph product $\mathscr{GP}(\Gamma,\mathcal{M})$  is isomorphic to the  restricted direct product of the monoids in $\mathcal{M}$, that is, to the submonoid of elements of the direct product with only finitely many non-identity values, and so to the direct product in the case where $|V|<\infty$.

\item[$\bullet$] If the monoids in $\mathcal{M}$ are all free monoids, then $\mathscr{GP}(\Gamma,\mathcal{M})$ 
is a {\em free partially commutative monoid}; if in addition $\Gamma$ is finite and each vertex monoid is finitely generated, then $\mathscr{GP}$ is a   {\em trace monoid}  (see, for example, \cite{diekert:1990}).

\item[$\bullet$] If the monoids in $\mathcal{M}$ are all free groups, then $\mathscr{GP}(\Gamma,\mathcal{M})$ is a {\em free partially commutative group}; if in addition $\Gamma$ is finite and each vertex group is finitely generated, then $\mathscr{GP}$ is a {\em right angled Artin group} (see, for example, \cite{duncan:2020}).

\item[$\bullet$] The graph product $\mathscr{GP}'=\mathscr{GP}(\Gamma', \mathcal{M}')$ of $\mathcal{M}'=\{M_\alpha: \alpha\in V'\}$ with respect to $\Gamma'=V\backslash \{\alpha\in V: |M_\alpha|=1\}$ is isomorphic to $\mathscr{GP}(\Gamma,\mathcal{M})$.

\item[$\bullet$] If $\Gamma'=\Gamma'(V', E')$ is  an induced subgraph of $\Gamma$, then the graph product $\mathscr{GP}'(\Gamma', \mathcal{M}')$ of $\mathcal{M}'=\{M_\alpha: \alpha\in V'\}$ with respect to $\Gamma'$ is a retract of $\mathscr{GP}(\Gamma,\mathcal{M})$ (see, \cite[Proposition 2.3]{gould:2023}). 

\item[$\bullet$] By taking $V'=\{\alpha\}$ above we have that $M_\alpha$ is naturally embedded in $\mathscr{GP}(\Gamma,\mathcal{M})$
under $a\mapsto [a]$ and the image is a retract of $\mathscr{GP}(\Gamma,\mathcal{M})$.
\end{enumerate}
\end{Rem}

For the remainder of this article  $\mathscr{GP}=\mathscr{GP}(\Gamma,\mathcal{M})$ will be given as in Definition~\ref{defn:graphprodmonoids} and we will  assume that each vertex monoid $M_\alpha$ $(\alpha\in V)$ is  non-trivial.

\subsection{Reduced words}\label{sub:reduced}

\begin{defn}\label{defn:shuffle} We refer to an application of a relation in $R_{\textsf{e}}$ to a word $w\in X^*$ as a {\em shuffle}. If $a\in M_{\alpha}$ and $b\in M_{\beta}$ and $(\alpha,\beta)\in E$, then we say that $a$ {\em shuffles with} $b$.
Two words in $X^*$ are {\it shuffle equivalent} if one can be obtained from the other by a series of shuffles.
\end{defn}

It is clear that shuffle equivalent words have the same length.

\begin{defn}\label{def of reduced} \label{defn:reducedform}  We refer to an application of a relation in $R$ to a word $w\in X^*$ that reduces
the length of $w$ as a {\em reduction}.
A word $w\in X^*$ is {\em reduced} if it is not possible to apply a reduction step to any word shuffle equivalent to $w$.

If $u\in X^*$ and $[u]=[w]$ for a reduced word $w\in X^*$, then we say that $w$ is a {\em reduced form} of $u$.
\end{defn}

\begin{Thm}\label{shuffle}\cite{fountain:2009} 
Every element of $\mathscr{GP}$ is represented by a reduced word. 
An element $x\in [w]$ is of minimal length if and only if it is reduced. Two reduced words represent the same element of $\mathscr{GP}$ if and only if they are shuffle equivalent.
\end{Thm}

\begin{defn}\label{defn:support}
Let $w=x_1\circ \dots \circ x_n\in X^*$ with $x_i\in M_{\alpha_i}$ for all $1\leq i\leq n$. We call the
set $\{\alpha_1, \dots, \alpha_n\}$ the  {\it support} of $w$ and denote it by $\supp(w)$.  If $\supp(w)=\{ \alpha\}$ is a singleton, then for the sake of notational convenience we may identify $\supp(w)$
with the vertex $\alpha$.
\end{defn}

It follows from the construction that a word  $w=x_1\circ \cdots \circ x_n\in X^*$ is  reduced if and only if for all $1\leq i\leq n$, $x_i\not \in I$, and for all $1\leq i< j\leq n $ with  $\supp(x_i)=\supp(x_j)$, there exists some $i<k<j$ with $(\supp(x_i), \supp(x_k))\not \in E$.

We now make a crucial observation, which stems from the fact that in any shuffle of a word, two letters with the same support cannot have their order transposed. 

\begin{Rem}\label{rem:allimportant} Let $x=x_1\circ \cdots \circ x_n\in X^*$ and let $\alpha\in \supp(x)$. Let  $ \{ i_1,\cdots, i_k\}$ label the positions where $\supp(x_i)=\alpha$,  where
 $i_1<i_2\cdots <i_k$. Then in any shuffle of $x$ the subword $x_{i_1}\circ\cdots\circ x_{i_k}$ will appear. Thus $x_{i_j}$ will always
be the $j$'th occurrence of a  letter with support $\alpha$ (reading from left to right) in any shuffle of $x$.\end{Rem}

\begin{Prop}\label{easy observation}
Let $x=x_1\circ \cdots \circ x_n\in X^*$ be reduced, and let $\sigma$ be a  permutation  of $[1, n]$. Then $x$  shuffles to $ y=x_{1\sigma}\circ \cdots \circ x_{n\sigma}$ if and only if for any $m<k$ and $k\sigma<m\sigma$ we have  $(\supp(x_{k\sigma}),\supp(x_{m\sigma}))\in E$.
\end{Prop}

\begin{proof}
Suppose that $x$ shuffles to $y=x_{1\sigma}\circ \cdots \circ x_{n\sigma}$. Notice that  $i\sigma=j$ indicates that the letter $x_j$ (in the $j$'th place of $x$) has shuffled in the shuffle from $x$ to $y$ to the $i$'th place of $y$. Assume that  $m<k$ and $k\sigma<m\sigma$.
In the shuffling process,  we must shuffle the letter $x_{m\sigma}$ to the $m$'th place, and the letter
$x_{k\sigma}$ to the $k$'th place. But as $k\sigma<m\sigma$ and $m<k$ we must in this process move the letter $x_{m\sigma}$
to the left of the letter $x_{k\sigma}$, which we can only do if $(\supp(x_{k\sigma}),\supp(x_{m\sigma}))\in E$.

We show the converse by induction on the length $|w|$ of $w$. The result holds vacuously when $|w|=1$.

Assume that the result holds for all $w$ with $|w|<n$. Suppose that for any $k>m$ and $k\sigma<m\sigma$ we have $(\supp(x_{k\sigma}),\supp(x_{m\sigma}))\in E$. If $n=n\sigma$, then by the inductive hypothesis $x_1\circ \cdots \circ x_{n-1}$ shuffles to $x_{1\sigma}\circ \cdots \circ x_{(n-1)\sigma}$, and so $w$ shuffles to $x_{1\sigma}\circ \cdots \circ x_{n\sigma}$.

If $n\sigma=\ell<n$, then there exists $1\leq d\leq n-1$ such that $d\sigma=n$. Notice that $d\sigma>h\sigma$ for all $h>d$, implying that  $(\supp(x_{d\sigma}), \supp(x_{h\sigma}))\in E$. Therefore, $x_{1\sigma}\circ \cdots \circ x_{n\sigma}$ shuffles to $x_{1\sigma}\circ \cdots\circ x_{(d-1)\sigma}\circ x_{(d+1)\sigma} \circ \cdots \circ  x_{n\sigma}\circ x_{d\sigma}$. We now define a permutation $\delta$ of $[1,n-1]$ by $t\delta=t\sigma$ for all $1\leq t\leq d-1$ and $t\delta=(t+1)\sigma$ for all $d\leq t\leq n-1.$ This gives $$x_{1\sigma}\circ \cdots\circ x_{(d-1)\sigma}\circ x_{(d+1)\sigma} \circ \cdots \circ  x_{n\sigma}=x_{1\delta}\circ \cdots\circ x_{(d-1)\delta}\circ x_{d\delta} \circ \cdots \circ  x_{(n-1)\delta}$$ as words.  Notice that for any $k>m$ and $k\delta<m\delta$ we have $(\supp(x_{k\delta}),\supp(x_{m\delta}))\in E$, and again by the inductive hypothesis we have
$x_1\circ \cdots \circ x_{n-1}$ shuffles to $x_{1\delta}\circ \cdots\circ x_{(d-1)\delta}\circ x_{d\delta} \circ \cdots \circ  x_{(n-1)\delta}$, and hence to $x_{1\sigma}\circ \cdots\circ x_{(d-1)\sigma}\circ x_{(d+1)\sigma} \circ \cdots \circ  x_{n\sigma}$. Therefore,
$x=x_1\circ \cdots \circ x_{n-1}\circ x_n$ shuffles to  $x_{1\sigma}\circ \cdots\circ x_{(d-1)\sigma}\circ x_{(d+1)\sigma} \circ \cdots \circ  x_{n\sigma}\circ x_{d\sigma}$, and hence to $y=x_{1\sigma}\circ \cdots\circ x_{(d-1)\sigma}\circ  x_{d\sigma} \circ x_{(d+1)\sigma} \circ \cdots \circ  x_{n\sigma}$.
\end{proof}

The following result is useful in later arguments.

\begin{lemma}\label{block}\cite[Lemma 3.14]{gould:2023}
Let $[x]=[y]$ where  $x=x_1\circ \cdots \circ x_n$ and $y=y_1 \circ \cdots \circ y_n$ are reduced and let $1\leq m\leq n$. Then $[x_1\circ \cdots \circ x_m]=[y_1\circ \cdots \circ y_m]$ if and only if $[x_{m+1}\circ \cdots \circ x_n]=[y_{m+1}\circ \cdots \circ y_n]$.
\end{lemma}

\subsection{Reducing products}\label{sub:reducing products}

As the paper progresses, we will need to present increasingly detailed results about the way in which the product of two reduced words reduces. We begin here with some known preliminaries, the first remark following easily from an earlier comment.

\begin{Rem}\label{cor:concatenate} Let $x=x_1\circ\cdots \circ x_m$ and $ y=y_1\circ\cdots \circ y_n\in X^*$ be reduced. Then $x\circ y$ is {\em not}
 reduced  exactly if there exist
$i,j$ with $1\leq i\leq m,1\leq j\leq n$ such that
$\supp(x_i)=\supp(y_j)$ and for all $h,k$ with $i<h\leq m,1\leq k<j$
we have $(\supp(x_i),\supp(x_h))\in E$ and $(\supp(y_j),\supp(y_k))\in E$.
\end{Rem}

\begin{lemma}\cite[Lemma 3.10]{gould:2023} \label{product reduction}
Let $p\in X$, where $p\not \in I$,  and let $a=a_1\circ \cdots \circ a_n\in X^*$ be  reduced. Then one of the following occurs:
\begin{enumerate}

\item[(i)]  $p\circ a$ is  reduced;

\item[(ii)] there exists $1\leq k\leq n$ such that $\supp(a_k)=\supp(p)$ and $(\supp(p), \supp(a_l))\in E$ for $1\leq l\leq k-1$, and  $p\circ a$ reduces to
shuffle equivalent words
\[pa_{k}\circ a_1\circ \cdots \circ a_{k-1}\circ a_{k+1}\circ \cdots \circ a_n\mbox{ and }a_1\circ \cdots \circ a_{k-1}\circ pa_k \circ a_{k+1},\]
such that

\item[(a)] if  $pa_k\notin I$   then $pa_{k}\circ a_1\circ \cdots \circ a_{k-1}\circ a_{k+1}\circ \cdots \circ a_n$ is  reduced;

\item[(b)] if  $pa_k\in I$ then $p\circ a$ reduces in one further step to the reduced word
\[a_1\circ \cdots \circ a_{k-1} \circ a_{k+1}\circ \cdots \circ a_n.\]
\end{enumerate}
\end{lemma}

Lemma~\ref{product reduction} tells us that if $a$ is reduced and $p\in X\setminus I$, then either $p\circ a$ is reduced or we need to handle $p$ exactly once: by shuffling and  applying $R_{\textsf{v}}$, or by shuffling  and applying $R_{\textsf{v}}$ and $R_{\textsf{id}}$ in turn.

\subsection{Foata normal form}

\begin{defn}
A reduced word $w\in X^*$ is called a {\it complete block} if the subgraph of $\Gamma$ induced by $\supp(w)$ is complete.
\end{defn}

A word $w\in X^*$ is a complete block if and only if $\supp(w)$ is a complete subgraph, no letters in $w$ are from $I$ and for each $\alpha\in \supp(w)$ there exists a unique letter contained in $w$ with support $\alpha$.

\begin{defn}\cite{gould:2023}\label{defn:lcrf} Let $w\in X^*$. Then $w$ is a {\it left Foata normal form}
with {\em block length $k$} and {\em blocks} $w_i\in X^*$, $1\leq i\leq k$, if:

\begin{enumerate}
\item [(i)] $w=w_1\circ \cdots \circ w_k\in X^*$ is a reduced word;

\item [(ii)] $\supp(w_i)$ is a complete subgraph for all $1\leq i\leq k$;

\item [(iii)]  for any $1\leq i<k$ and  $\alpha\in \supp(w_{i+1})$, there is some $\beta\in \supp(w_i)$ such that $(\alpha, \beta)\not\in E$.

\end{enumerate}
If $[u]=[w]$ where $w$ is a left  Foata normal form, then 
$w$ is a  {\em  left  Foata normal form of $u$}. \end{defn}

The notion of {\it right Foata normal form} is defined in a dual way to that of left Foata normal form.

Given a reduced word $w\in X^*$ we may shuffle $w$ to $w_1\circ w'$ where $w_1$ is the longest possible complete
block. We apply the same shuffling process to $w'$ and continue until $w$ is a product of complete blocks. That this process yields the same blocks is the content of the next result.

\begin{Thm}\cite[Proposition 3.18, Theorem 3.20]{gould:2023}\label{thm:uniqueness} Every  element $w$ of $X^*$ has a left Foata normal form  $w_1\circ w_2\circ\dots\circ w_k$,
with blocks $w_i$ for $1\leq i\leq k$.

If  $w_1'\circ w_2'\circ\dots \circ w_h'$ is any left Foata normal form of $w$ with blocks $w_j'$ for $1\leq j\leq h$, then $k=h$ and $[w_i]=[w_i']$
for $1\leq i\leq k$.
\end{Thm}

The next result is folklore, but follows easily from uniqueness in Theorem~\ref{thm:uniqueness}.

\begin{Cor}\label{cor:firstblock} Let $w=x_1\circ\cdots\circ x_n\in X^*$ have left Foata normal form  $w_1\circ w_2\circ\dots\circ w_k$,
with blocks $w_i$ for $1\leq i\leq k$. Then $w_1$ consists of all the letters $x_i$ such that $(\supp(x_i),\supp(x_j))\in E$ for all $1\leq j<i$.
\end{Cor}

It can be useful to refer to the letters in $w_1$ above as `the letters that shuffle to the front of $w$', particularly if we do not require
the full labelling of left Foata normal form.

\section{The ascending chain condition on principal left ideals}\label{sec:AAPL}

Recall that a monoid $S$ satisfies {\it ascending chain condition on principal left ideals}, abbreviated by ACCPL,  if  every chain  $$Sa_1 \subseteq Sa_2\subseteq\cdots$$ 
stabilises, that is, there is an $n\geq 1$ such that 
$Sa_n=Sa_{n+k}$ for all $k\geq 1$. Stopar \cite{stopar:2012}  showed that the class of monoids with ACCPL is closed under finite direct products.  Miller reproved this and also showed the corresponding result for free products \cite{miller:2023}. The aim of this section is to prove that a  graph product $\mathscr{GP}$ has ACCPL if and only if  each vertex monoid has ACCPL.

\begin{lemma}\label{prop:retract} Let $T$ be a retract of a monoid $S$. Then for any $a,b\in T$ we have
\[Ta\subseteq Tb\mbox{ if and only if }Sa\subseteq Sb.\]\end{lemma}
\begin{proof} Let $S$ be a monoid with retract $T$, that is, $T$ is a submonoid of $S$ and there is an onto morphism $\theta:S\rightarrow T$ that
restricts to the identity on $T$. For any monoid $M$ and $u,v\in M$ we have  $Ma\subseteq Mb$ if and only if $a=sb$ for some $s\in M$.
Clearly then if $a,b\in T$ and $Ta\subseteq Tb$, then $Sa\subseteq Sb$. On the other hand, if $Sa\subseteq Sb$, then choosing $t\in S$ with
$a=tb$ and applying the retraction $\theta$, we have $a=(t\theta) b$, so that $Ta \subseteq Tb$. 
\end{proof}

\begin{corollary}\label{prop:retract} The class of monoids satisfying ACCPL is closed under retract.
\end{corollary}

We have observed in Remark~\ref{rem:freeprod} that each vertex monoid $M_\alpha$ naturally embeds into the graph product $\mathscr{GP}$, and moreover
there is a retraction from  $\mathscr{GP}$ to $M_{\alpha}$. These considerations immediately yield the following.

\begin{corollary}\label{cor:oneway}  If $\mathscr{GP}$ has ACCPL, then so does each vertex  monoid $M_\alpha$.
\end{corollary}

We aim to establish the converse to Corollary~\ref{cor:oneway}. To do so,  we  characterise the inclusion relation between principal left ideals  in $\mathscr{GP}$. We begin by investigating the inclusion $\mathscr{GP}[u] \subseteq \mathscr{GP}[v]$ for $u, v \in X^*$ under the additional constraint that $v$ contains no left invertible letters that can be shuffled to the front, or equivalently, the first block of a left Foata normal form of $u$ contains no left invertible letters.

\begin{Lem}\label{standard multi}
Suppose $u=u_1\circ \cdots \circ u_m\in X^*$ is a left Foata normal form with block length $m$ and blocks $u_i$, $1\leq i\leq m$,  such that $u_1$ contains no left invertible letters, and let $a\in X^*$ be a reduced word. Let $z$ be a reduced form of $a\circ u_1$. Then $z\circ u_2\circ \cdots \circ u_m$ is a reduced  form of $a\circ u$. Further, $|z|\geq |u_1|$ (and hence $|z\circ u_2\circ \cdots \circ u_m|\geq |u|$), and $|z|=|u_1|$ (and hence $|z\circ u_2\circ \cdots \circ u_m|=|u|$) if and only if $\supp(a)\subseteq \supp(u_1)$.
\end{Lem}
\begin{proof}
It follows from \cite[Lemma 4.4]{gould:2023} that $z\circ u_2\circ \cdots \circ u_m$ is a reduced form of $a\circ u$. Since  $u_1$ contains no left invertible letters, the support of the reduced form $z$ of $a\circ u_1$ must contain $\supp(u_1)$, implying $|z|\geq |u_1|$.
It is clear from repeated applications of Lemma~\ref
{product reduction} that $\supp(a)\subseteq \supp(u_1)$ if and only if $|z|=|u_1|$. 
\end{proof}

\begin{lemma}\label{key7}\label{key4}
Let $a, b\in X\setminus I$ be such that $b$ is not left invertible. Then $\mathscr{GP} [a]\subseteq \mathscr{GP}[b]$ if and only if  $\supp(a)=\supp(b)$ and $M_{\mathbf{s}(a)} a\subseteq M_{\mathbf{s}(b)} b$. 

Consequently, if also $a$ is not left invertible, then  $\mathscr{GP} [a]= \mathscr{GP}[b]$ if and only if  $\supp(a)=\supp(b)$ and $M_{\mathbf{s}(a)} a=M_{\mathbf{s}(b)} b$.
\end{lemma}

\begin{proof}
Assume $\mathscr{GP}[a] \subseteq \mathscr{GP}[b]$. Then there exists a reduced word $c \in X^*$ such that $[a] = [c][b]$.  By Lemma \ref{standard multi}, we have $\supp(c) = \supp(b)$, and consequently, $[a] = [cb]$. Applying Remark \ref{rem:freeprod}, we deduce that $a = cb$ holds in $M_{\mathbf{s}(a)}$, and thus 
$
M_{\mathbf{s}(a)} a \subseteq M_{\mathbf{s}(b)} b.$ 

The proof of the converse is direct.
\end{proof}

\begin{Lem}\label{key8}
Let $a=a_1\circ \cdots \circ a_n\in X^*$ be a complete block containing no left  invertible letters. Then $$\mathscr{GP}[a]=\bigcap_{1\leq i\leq n}\mathscr{GP}[a_i].$$
\end{Lem}

\begin{proof}
Since $a=a_1\circ \cdots \circ a_n$ is complete, we can easily deduce $\mathscr{GP}[a]\subseteq \mathscr{GP}[a_i]$ for each $1\leq i\leq n$, and thus $$\mathscr{GP}[a]\subseteq \bigcap_{1\leq i\leq n}\mathscr{GP}[a_i].$$

Conversely, let $[z]\in \bigcap_{1\leq i\leq n}\mathscr{GP}[a_i]$. Then there exist reduced words $w_1, \cdots, w_n\in X^*$ such that $$[z]=[w_1][a_1]=\cdots=[w_n][a_n].$$ For each $1\leq i\leq n$, let $w_i=x^i_1\circ \cdots \circ x^i_{m_i}$. Then, by Lemma \ref{product reduction}, either $x^i_1\circ \cdots \circ x^i_{m_i}\circ  a_i$ is reduced or it reduces to $$x^i_1\circ \cdots \circ x^i_{j-1}\circ x^i_{j+1} \circ \cdots \circ x^i_{m_i}\circ x_j^ia_i$$ for some $1\leq j\leq m_i$. This implies either $a_i$ or $x^i_{j}a_i$ is in the last block of a right Foata normal form of $w_i\circ a_i$, and hence of $z$. Therefore, 

\[[z]=[u][p][a_1\circ \cdots\circ a_n]\]
where $p$ is the product (in any order) of the $x^i_{j}$ (where they exist). This gives that $$[z]\in \mathscr{GP}[a_1\circ \cdots\circ a_n]=\mathscr{GP}[a].\qedhere $$\end{proof}

\begin{Lem}\label{left standard}
Let $a=a_1\circ \cdots \circ a_m$ and $ b=b_1\circ \cdots \circ b_n\in X^*$  be complete blocks such that neither $a$ nor $b$ contain any left invertible letters. Then $\mathscr{GP}[a]\subseteq \mathscr{GP}[b]$ if and only if 
\[\supp(b)\subseteq \supp(a)\]
and with suitable relabelling of the letters of $a$ we have
\[\mathscr{GP}[a_i]\subseteq \mathscr{GP}[b_i]\]
 for each $1\leq i\leq n\leq m$. 
\end{Lem}

\begin{proof}
Assume $\mathscr{GP}[a]\subseteq \mathscr{GP}[b]$. Then there exists a reduced word $c\in X^*$ such that $[a]=[c][b]$. Since $b$ contains no left invertible letters, it follows from Lemma \ref{product reduction} that $c\circ b$ reduces to  a word $$w=q\circ p_1b_1\circ \cdots \circ p_nb_n$$ where $c=q\circ p_1\circ\cdots \circ p_n$ (as a product in $X^*$) and $p_i$ is either some letter from $c$ or $\epsilon$, for each $1\leq i\leq n$. Consequently, $$[a]=[w]\mbox{ and } \supp(w)=\supp(a)\supseteq \supp(b).$$ Relabelling the elements of $a$ 
so that $\supp(a_i)=\supp(b_i)$ for $1\leq i\leq n$ we have that $a_i=p_ib_i$  and thus $\mathscr{GP}[a_i]\subseteq \mathscr{GP}[b_i]$ for each 
$1\leq i\leq n$.

Conversely, if $\supp(b)\subseteq \supp(a)$ and $\mathscr{GP}[a_i]\subseteq \mathscr{GP}[b_i]$, then, by Lemma \ref{key8}, we have $$\mathscr{GP}[a_1\circ  \cdots \circ a_n]=\bigcap_{1\leq i\leq n}\mathscr{GP}[a_i] \subseteq \bigcap_{1\leq i\leq n}\mathscr{GP}[b_i]=\mathscr{GP}[b_1\circ  \cdots \circ b_n]=\mathscr{GP}[b].$$
Furthermore, since $a$ is a complete block, it follows that $$\mathscr{GP}[a]=\mathscr{GP}[a_{n+1}\circ \cdots a_m\circ a_1\circ \cdots \circ a_n]\subseteq \mathscr{GP}[a_1\circ \cdots \circ a_n],$$ so that  $\mathscr{GP}[a]\subseteq \mathscr{GP}[b]$, as required.
\end{proof}

\begin{Cor}\label{cor-left standard}
Let $a=a_1\circ \cdots \circ a_m$ and $ b=b_1\circ \cdots \circ b_n\in X^*$ be complete blocks such that both $a$ and  $b$ contain no left invertible letters. Then $\mathscr{GP}[a]=\mathscr{GP}[b]$ if and only if \[\supp(a)=\supp(b)\]
and with suitable relabelling of the letters of $a$ we have
\[\mathscr{GP}[a_i]=\mathscr{GP}[b_i]\]
 for each $1\leq i\leq n=m$. 
\end{Cor}

\begin{Lem}\label{key9}
Suppose each vertex monoid satisfies ACCPL. Let $\{w_i: i\in \mathbb N\}$ be a set of complete blocks with the same length $n$, containing no left invertible letters. Then the following chain $$\mathscr{GP}[w_1]\subseteq \mathscr{GP}[w_2]\subseteq \cdots$$ of principal left ideals  terminates.
\end{Lem}

\begin{proof}
By Lemma \ref{left standard}, we have $\supp(w_i)=\supp(w_j)$ for all $i, j\in \mathbb N$. Let $w_i=x^i_{1}\circ \cdots \circ x^i_{n}\in X^*$ for each $i\in \mathbb N$. Consider $w_1=x^1_{1}\circ \cdots \circ x^1_{n}$ and let $\supp(x^1_{k})=\alpha_k$ for each $1\leq k\leq n$. By Lemma \ref{left standard} and with suitable relabelling of the letters of $w_i$ for all $i\leq 2$,  we have the following $n$ chains $$\mathscr{GP}[x^1_{i}]\subseteq \mathscr{GP}[x^2_i]\subseteq \cdots$$
where $1\leq i\leq n$. Furthermore, it follows from Lemma \ref{key7} that for  $1\leq i\leq n$ we have
$$M_{\alpha_i}x^1_{i}\subseteq M_{\alpha_i}x^2_i\subseteq \cdots$$

Since each vertex monoid satisfies the ascending chain condition on principal left ideals, there exist $t_i$, where $1\leq i\leq n$, such that 
$M_{\alpha_i}x^{t_i}_{i}=M_{\alpha_i}x^{t_i+1}_{i}=\cdots.$ Let $d$ be the maximum of $t_i$, where $1 \leq i\leq n$.  Then 
$$M_{\alpha_i}x^{d}_{i}=M_{\alpha_i}x^{d+1}_{i}=\cdots$$
for $1\leq i\leq n$. It follows from Lemma~\ref{key4} that 
$$\mathscr{GP}[x^d_{i}]=\mathscr{GP}[x^{d+1}_{i}]=\cdots$$
again for $1\leq i\leq n$. Consequently, 
by Corollary \ref{cor-left standard}, we have $$\mathscr{GP}[w_d]=\mathscr{GP}[w_{d+1}]=\cdots$$ implying the chain $$\mathscr{GP}[w_1]\subseteq \mathscr{GP}[w_2]\subseteq \cdots$$ terminates, as required. 
\end{proof}

\begin{lemma}\label{key5}

For each word $w\in X^*$, there exists a reduced word $u\in X^*$ such that $\mathscr{GP} [w]=\mathscr{GP}[u]$, where $u$ has no left invertible elements that can be shuffled to the front.
\end{lemma}

\begin{proof}
It follows from  \cite[Lemma 4.2]{gould:2023} that $[w]=[a][u]$ where $a\circ u$ is reduced,  $[a]$ is left invertible and $u$ has no left invertible elements that can be shuffled to the front. This implies $\mathscr{GP}[w]=\mathscr{GP}[u]$. 
\end{proof}

\begin{Lem}\label{key6}
Let $u=u_1\circ \cdots \circ u_m, v=v_1\circ \cdots \circ v_n\in X^*$ be left Foata normal forms with blocks $u_i, v_j$, where  $1\leq i\leq m, 1\leq j\leq n$. Suppose $v_1$ contains no left invertible letters and $\mathscr{GP}[u]\subseteq \mathscr{GP}[v]$. Then one of the following holds
\begin{enumerate}
\item[(i)] $|u|>|v|$;

\item[(ii)] $|u|=|v|$, $m=n$, $\mathscr{GP}[u_1]\subseteq \mathscr{GP} [v_1]$ and $[u_i]=[v_i]$ for all $2\leq i\leq m$.
\end{enumerate}
\end{Lem}
\begin{proof}
Since $\mathscr{GP}[u]\subseteq \mathscr{GP}[v]$, there exists a reduced word $c\in X^*$ such that $[u]=[c][v]$. From Lemma \ref{standard multi} it follows that $|u|\geq |v|$. Suppose $|u|=|v|$. Again, by Lemma \ref{standard multi}, $\supp(c)\subseteq \supp(v_1)$ . Let $z$ be a reduced form of $c\circ v_1$. Then $z$ is a complete block and $\supp(z)=\supp(v_1)$. It follows that $z\circ v_2\circ \cdots \circ v_n$ is a left Foata normal form of $c\circ v$. By Lemma \ref{thm:uniqueness} this implies  that $m=n$, $[u_1]=[z]=[c][v_1]$, and so $\mathscr{GP}[u_1]\subseteq \mathscr{GP} [v_1]$, and $[u_i]=[v_i]$ for all $2\leq i\leq m$.
\end{proof}

It is clear that if  $\mathscr{GP}[u]\subseteq \mathscr{GP} [v]$ then $\mathscr{GP}[u][w]\subseteq \mathscr{GP} [v][w]$ for any
$u,v,w\in X^*$.

\begin{Cor}
Let $u=u_1\circ \cdots \circ u_m$ and $v=v_1\circ \cdots \circ v_n\in X^*$ be left Foata normal forms with blocks $u_i, v_j$, for $1\leq i\leq m, 1\leq j\leq n$, such that $u_1$ and $v_1$ contain no left invertible letters. Then $\mathscr{GP}[u]=\mathscr{GP}[v]$ if and only if $$|u|=|v|, \quad m=n, \quad \mathscr{GP}[u_1]=\mathscr{GP} [v_1], \quad [u_i]=[v_i]$$ for each $2\leq i\leq m.$
\end{Cor}
\begin{Thm}\label{acc}
A graph product $\mathscr{GP}=\mathscr{GP}(\Gamma,\mathcal{M})$ satisfies the ascending chain condition on  principal left ideals if and only if each so does every vertex monoid.
\end{Thm}
\begin{proof} One direction is  Corollary~\ref{cor:oneway}.

For the converse, suppose  that each vertex monoid satisfies the ascending chain condition on  principal left ideals. Let $\{w_i\in X^*: i\in \mathbb{N}\}$ be a set of reduced words and  $$\mathscr{GP}[w_1]\subseteq \mathscr{GP}[w_2]\subseteq \cdots$$   a chain of principal left ideals of $\mathscr{GP}$. By Lemma \ref{key5}, we may assume each $$w_i=p^i_1\circ \cdots \circ p^i_{\lambda(i)}$$ is in left Foata normal form  with block length $\lambda(i)$ and blocks $p^i_j\in X^*$, $1\leq j\leq \lambda(i)$, and  that the first block $p^i_1$ contains no left invertible letters.

By Lemma~\ref{key6} we have that 
\[|p_1^1|\geq |p_1^2|\geq \cdots.\]
Let $n$ be such that $|p_1^n|=|p_1^{n+k}|$ for all $k\geq 0$. Again by Lemma~\ref{key6} we have that
$\lambda(n)=\lambda(n+k)$ and  $ [p^n_{j}]=[p^{n+k}_{j}]$ for all $k\geq 0$ and $2\leq j\leq \lambda(n)$, 
and further,
\[\mathscr{GP}[p_1^n]\subseteq \mathscr{GP}[p_1^{n+1}]\subseteq\cdots.\] 
From Lemma~\ref{key9} this ascending chain terminates. Putting  
 $w=p^n_{2}\circ\cdots \circ p^n_{\lambda(n)}$ and noting that for $i\geq n$ we have $[w_i]=[p_1^i][w]$ it follows that 
 \[\mathscr{GP}[w_n]\subseteq \mathscr{GP}[w_{n+1}]\subseteq\cdots,\]
and hence our original ascending chain, terminates. 
\end{proof}

\begin{Cor} cf. \cite{miller:2023,stopar:2012}
A free product or restricted direct product of an indexed set of  monoids $\{ M_{i}:i\in I\}$ satisfies the ascending chain condition on principal left ideals if and only if so does each $M_i,i\in I$.
\end{Cor}

\section{Weakly left noetherian}\label{sec:wln}

A monoid $M$ is \emph{weakly left noetherian} if it has the ascending chain condition on left ideals; equivalently, every left ideal is finitely generated (as a left ideal).  From \cite{miller:2021}, the monoid $M$ is weakly left noetherian if and only if it satisfies the ascending chain condition on principal left ideals and contains no
 infinite anti-chain of principal left ideals (under $\subseteq$).  Being weakly left noetherian is therefore a strictly stronger property than having the ascending chain condition on principal left ideals, as may easily be seen by considering
an infinite null semigroup with an identity adjoined. On the other hand, it is not as strong as that of being \emph{left noetherian}, which states that the ascending chain condition on  the lattice of left congruences holds (equivalently, every left congruence is finitely generated). We remark that in the context of rings,  the analogous properties of being weakly left noetherian and left noetherian coincide and imply \emph{left coherency}. Weakly left noetherian monoids have been extensively investigated, see, for example, \cite{hotzel:1976, miller:2021, sat:1972}.

In this section we consider  the property of being weakly left noetherian in the context of graph products. As presented in  Theorem~\ref{wln}, we will see that it is somewhat less amenable in this regard than that of the ascending chain condition on principal left ideals. 

In the positive direction, \cite[Lemma 4.1]{miller:2021} gives that the class of weakly left noetherian monoids is closed under morphic image, and so certainly under retract. Note that the class of  monoids satisfying the ascending chain condition on principal left ideals  is not closed under morphic image, which is clear from considering the fact that any free monoid satisfies this  this condition.

\begin{corollary}\label{cor:retract3} If $\mathscr{GP}$ is weakly left noetherian, then so is each vertex monoid.
\end{corollary}

Unlike the case for the ascending chain condition on principal left ideals, the converse to Corollary~\ref{cor:retract3} does not hold in general. This  follows from the known  result  for free products of monoids.

\begin{Lem}\label{key13}\cite{miller:2021}
Let $M$ and $N$ be non-trivial monoids. 
\begin{enumerate}
\item[(i)] The free product  of $M$ and $N$ is weakly left noetherian if and only if either  $|M|=|N|=2$ or both $M$ and $N$ are groups.
    \item[(ii)] The direct product  of  $M$ and $N$ is weakly left noetherian if and only if  both $M$ and $N$ are weakly left noetherian.
    \end{enumerate}
\end{Lem}

In our work there is no assumption that the graph $\Gamma$ is finite. We therefore need an infinitary version of Lemma~\ref{key13}(ii). It is helpful to recall that a monoid in which every element is left invertible is a group.

\begin{Prop}\label{prop:reducedsimple} Let $\mathcal{M}=\{ M_i:i\in I\}$ be a collection of an indexed set of monoids and let
$P$ be the restricted direct product of the monoids in  $\mathcal{M}$. Then $P$ is weakly left noetherian if and only if each $M_i$ is weakly left noetherian, and
$J:=\{ i\in I:M_i\mbox{ is not a group}\}$ is finite.\end{Prop}
\begin{proof} If $P$ is weakly left noetherian, then so is each $M_i$, by Corollary~\ref{cor:retract3}. Suppose that 
$J$ is infinite. For each $j\in J$ we pick an element $a_j\in M_j$ such that $a_j$ is not left invertible and define
$f_j\in P$ by $jf_j=a_j$ and $if_j=I_{M_i}$ for all $i\neq j$. Clearly $F=\{ f_j:j\in J\}\subseteq P$. 
Let $I$ be the left ideal of $P$ generated by $F$. Since $P$ is weakly left noetherian, $I$ is generated by a finite subset $F'$ of $F$.
Pick $f_k\notin F'$; by assumption, $f_k=g f_j$ for some $g\in P$ and $f_j\in F'$, and so in particular, $k\neq j$. Then
\[I_{M_j}=jf_k=j(g f_j)=(jg)(jf_j)=(jg)a_j,\]
contradicting the fact that $a_j$ is not left invertible. Thus $J$ is finite.

For the converse, suppose that each $M_i$ is weakly left noetherian, and
$J:=\{ i\in I:M_i\mbox{ is not a group}\}$ is finite. Then $P\cong P_J\times P_G$, where $P_J$ is the direct product of the monoids in $\mathcal{M}_J:=\{ M_j:j\in J\}$ and
$P_G$ is the restricted direct product of the groups in   $\mathcal{M}_G:=\{ M_j: j\in I\setminus J\}$. Certainly  $P_G$ is a group, and so is  weakly left noetherian. The result then follows from Lemma~\ref{key13}. \end{proof}

To identify which graph products are weakly left noetherian we introduce the notion of a graph   $\Gamma = (V, E)$ being  relatively complete with respect to a set of vertex monoids $\mathcal{M}$, as follows.

\begin{defn}\label{d10}
A  graph $\Gamma = (V, E)$ is \emph{relatively complete with respect to $\mathcal{M} = \{M_\alpha : \alpha \in V\}$} if for any pair $(\alpha, \beta) \notin E$ with $\alpha \neq \beta$, one of the following holds:
\begin{enumerate}
    \item[(i)] $|M_\alpha| = |M_\beta| = 2$, either $M_\alpha$ or $M_\beta$ is not a group, and $(\alpha, \gamma), (\beta, \gamma) \in E$ for all $\gamma \in V \setminus \{\alpha, \beta\}$;

    \item[(ii)] both $M_\alpha$ and $M_\beta$ are groups.
\end{enumerate}
\noindent  Moreover, there are only finitely many pairs $(\alpha, \beta) \notin E$ satisfying Condition (i).
\end{defn}

  Note  that if $(\alpha,\beta)$ and $(\gamma,\delta)$ are pairs of vertices satisfying Condition (i) of 
Definition~\ref{d10}, then $\alpha,\beta, \gamma$ and $\delta$ are all distinct.

Before proceeding with our arguments,   we make two straightforward observations. 
\begin{Rem}\label{pre1}
\begin{enumerate} 

\item[(i)] Suppose that $\Gamma' = (V', E')$ is an induced subgraph of $\Gamma = (V, E)$, and let $\mathscr{GP}' = \mathscr{GP}(\mathcal{M}', \Gamma')$ be the graph product of $\mathcal{M}' = \{M_\alpha : \alpha \in V'\}$ with respect to $\Gamma'$. We denote an  
element of $\mathscr{GP}'$ using the notation $\lfloor \cdot \rfloor$. Using  Remark~\ref{rem:freeprod},  for any  words $x, y, z \in X^*$ consisting of letters only from $\mathcal{M}'$, we have that $\lfloor x \rfloor = \lfloor y \rfloor$ in $\mathscr{GP}'$ if and only if $[x] = [y]$ in $\mathscr{GP}$; furthermore, $\lfloor x \rfloor \lfloor y \rfloor = \lfloor z \rfloor$ in $\mathscr{GP}'$ if and only if $[x][y] = [z]$ in $\mathscr{GP}$.

\item[(ii)] Suppose that $x,y \in X^*$ are shuffle equivalent. Let $x'$ be a subword of $x$ which shuffles (within $x$) to the subword $y'$ of $y$. Then $[x']=[y']$. Consequently, if
$x^*$ and $y^*$ are the subwords of $x$ and $y$ obtained by deleting $x'$ and $y'$ respectively, then also $[x^*]=[y^*]$.
\end{enumerate}
\end{Rem}

 Lemma \ref{decomp} will be crucial for our discussion. It  enables us to decompose $\mathscr{GP} = \mathscr{GP}(\Gamma,\mathcal{M})$, where $\Gamma$ is relatively complete with respect to $\mathcal{M}$, into a direct product of two or three monoids, each of which is weakly left noetherian,   as  in Proposition \ref{direct4}. 

\begin{Lem}\label{decomp}
Suppose $V$ is a disjoint union of $V_1$ and $V_2$ such that $(\alpha, \beta)\in E$ for all $\alpha\in V_1$ and all $\beta \in V_2$.  For $i=1,2$, let $\mathscr{GP}_i$ be the graph product of monoids $\mathcal{M}_i=\{M_\alpha: \alpha\in V_i\}$ with respect to the induced subgraph $\Gamma_i=(V_i, E_i)$ of $\Gamma$.Then $$\mathscr{GP}\cong \mathscr{GP}_1\times \mathscr{GP}_2.$$
\end{Lem}

\begin{proof}
Let $x = x_1 \circ \cdots \circ x_n \in X^*$ be reduced. Since $(\alpha, \beta) \in E$ for all $\alpha \in V_1$ and $\beta \in V_2$, we may shuffle $x$ to $x'\circ x^*$, so that  $[x] = [x'][x^*]$, where $\supp(x')\subseteq V_1$ and $\supp(x^*)\subseteq V_2$.
Suppose now $y \in X^*$ is reduced and shuffle equivalent to $x$. Then $[y] = [y'][y^*]$, where $y'$ and $y^*$ are defined analogously to $x'$ and $x^*$ for $x$. From Remark~\ref{pre1} we have  that $x'$ is shuffle equivalent to $y'$ and $x^*$ is shuffle equivalent to $y^*$, implying    $$[x'] = [y']\mbox{~and~}[x^*] = [y^*].$$ Let $\lfloor \cdot \rfloor$ and $\lceil \cdot \rceil$ denote  equivalence classes of words in $\mathscr{GP}_1$ and $\mathscr{GP}_2$, respectively. Then $\lfloor x' \rfloor=\lfloor y' \rfloor$ and $\lceil x^* \rceil=\lceil y^* \rceil$, again by Remark \ref{pre1}.  

We now define a map
\[
\psi: \mathscr{GP} \longrightarrow \mathscr{GP}_1 \times \mathscr{GP}_2, \quad [x] \mapsto \big(\lfloor x' \rfloor, \lceil x^* \rceil\big),
\]
where $x \in X^*$ is reduced. It follows from  above arguments and observation that $\psi$ is well-defined and injective. Clearly, $\psi$ is onto.

 We now prove $\psi$ is a homomorphism. Let $a, b \in X^*$ be reduced. Then $[a] = [a'][a^*]$ and $[b] = [b'][b^*]$, where $a'$ and $a^*$, $b'$ and $b^*$ are defined analogously to $x'$ and $x^*$ for $x$. Then, since $(\alpha, \beta) \in E$ for all $\alpha \in V_1$ and $\beta \in V_2$, we deduce
\[
[a][b] = [a'][a^*][b'][b^*] = [a'][b'][a^*][b^*] =[a'\circ b'][a^*\circ b^*].
\]
Let $p$ and $q$ be reduced forms of $a' \circ b'$ and $a^* \circ b^*$, respectively. Since $\supp(p) \subseteq V_1$ and $\supp(q) \subseteq V_2$, we have $\supp(p) \cap \supp(q) = \emptyset$, which implies that $p \circ q$ is a reduced form of $(a' \circ b') \circ (a^* \circ b^*)$, and hence of $a \circ b$. Therefore,
\[
([a][b])\psi = ([a \circ b])\psi = ([p \circ q])\psi = \big(\lfloor p \rfloor, \lceil q \rceil\big) = \big(\lfloor a' \circ b' \rfloor, \lceil a^* \circ b^* \rceil\big).
\]
Since $\lfloor a' \circ b' \rfloor = \lfloor a' \rfloor \lfloor b' \rfloor$ and $\lceil a^* \circ b^* \rceil = \lceil a^* \rceil \lceil b^* \rceil$, it follows that
\[
([a][b])\psi = \big(\lfloor a' \rfloor, \lceil a^* \rceil\big)\big(\lfloor b' \rfloor, \lceil b^* \rceil\big) = [a]\psi[b]\psi.
\]
Therefore, $\psi$ is a homomorphism and, consequently, an isomorphism.\end{proof}

\begin{Prop}\label{direct4}
Let $\Gamma=(V,E)$ be  relatively complete  with respect to $\mathcal{M}$.   Suppose that  there are precisely $n$ pairs of vertices 
$(\alpha_i, \beta_i)\notin E$ satisfying Condition {\rm(i)} of Definition \ref{d10}, where $1\leq i\leq n$ and $n\in \mathbb{N}^0$.
 Let  $V_1=\{\alpha_i, \beta_i: 1\leq i\leq n\}$, $$V_2=\{\alpha\in V: M_\alpha \mbox{~is~not~a~group}\}\backslash V_1 \quad \mbox{and} \quad V_3=V\backslash (V_1\cup V_2).$$

For $i=1,2,3$ let $\mathscr{GP}_i$ be the graph product of $\mathcal{M}_i=\{M_\alpha: \alpha\in V_i\}$ with respect to the induced subgraph $\Gamma_i=(V_i, E_i)$. Then $$\mathscr{GP}\cong \mathscr{GP}_1\times \mathscr{GP}_2\times \mathscr{GP}_3.$$ 
Further, $\mathscr{GP}_1$ is the direct product of the free products  $M_{\alpha_i}\ast M_{\beta_i}$ of the monoids $M_{\alpha_i}$ and $M_{\beta_i}$  where $1\leq i\leq n$,  $\mathscr{GP}_2$ is the restricted direct product of monoids in $\mathcal{M}_2$, and $\mathscr{GP}_3$ is the graph product of groups in $\mathcal{M}_3$. 
 Note that if $n=0$ then  the factor $\mathscr{GP}_1$ is trivial and so \[\mathscr{GP}\cong \mathscr{GP}_2\times \mathscr{GP}_3.\]

\end{Prop}

\begin{proof}
 If  $n>0$, that is, $V_1\neq \emptyset$, let  $V'=V\backslash V_1$. By the assumed condition we have $(\alpha, \beta)\in E$ for all $\alpha\in V_1$ and $\beta\in V'$, so that  by Lemma \ref{decomp} $$\mathscr{GP}\cong \mathscr{GP}_1\times \mathscr{GP}',$$ where $\mathscr{GP}'=\mathscr{GP}(\mathcal{M}', \Gamma')$ is the graph product of monoids $\mathcal{M}'=\{M_\alpha: \alpha\in V'\}$ with respect to the induced subgraph $\Gamma'=(V', E')$ of $\Gamma$. 

  By construction, $V’$ is the disjoint union of $V_2$ and $V_3$. It follows from the definition of a relatively complete graph that $(\alpha, \beta)\in E$  for all  distinct $\alpha,\beta\in V_2$ and for all $\alpha\in V_2$ and $\beta\in V_3$, giving $$\mathscr{GP}'\cong \mathscr{GP}_2\times \mathscr{GP}_3$$ by Lemma \ref{decomp}, and hence  $$\mathscr{GP}\cong \mathscr{GP}_1\times \mathscr{GP}_2\times \mathscr{GP}_3,$$
with the first factor trivial in the case $V_1=\emptyset$.

  It follows from the comment after Definition~\ref{d10} that the vertices $\alpha_1,\beta_1, \cdots, \alpha_n, \beta_n$ are all distinct. Moreover, from that definition,
$(\alpha_i,\gamma), (\beta_i,\gamma)\in E$ for all $1\leq i\leq n$ and all
$\gamma\in \{ \alpha_j,\beta_j:1\leq j\leq n, i\neq j\}$. It follows from repeated applications of Lemma~\ref{decomp}  that $$\mathscr{GP}_1\cong \mathscr{GP}_{11} \times \cdots \times \mathscr{GP}_{1n}$$ where   $\mathscr{GP}_{1i}$ is the 
free product  $M_{\alpha_i}\ast M_{\beta_i}$ for all $1\leq i\leq n$.

As $\Gamma_2$ is complete,  we have that  $\mathscr{GP}_2$ is the restricted direct product of monoids in $\mathcal{M}_2$. Clearly, all vertex monoids in $M_3$ are groups, and thus $\mathscr{GP}_3$ is the graph product of groups in $\mathcal{M}_3$.
\end{proof}

\begin{Prop}\label{key15}
If the graph product $\mathscr{GP}=\mathscr{G}(\Gamma,\mathcal{M})$ is weakly left noetherian, then $\Gamma$ is relatively complete with respect to 
$\mathcal{M}$. 
\end{Prop}

\begin{proof}
Suppose there exists a pair $(\alpha, \beta) \notin E$ with $\alpha \neq \beta$ and either $M_\alpha$ or $M_\beta$ is not a group. Without loss of generality, assume  $M_\alpha$ is not a group. We now claim condition \rm(i) of Definition \ref{d10} holds.

{\bf Step 1} We first prove $|M_\alpha| = |M_\beta| = 2.$
Let $\Gamma'=(V', E')$ be the induced subgraph of $\Gamma=(V, E)$ where $V'=\{\alpha, \beta\}$ and let
$\mathcal{M}'=\{ M_\alpha,M_\beta\}$.  By Remark~\ref{rem:freeprod} we have that
$\mathscr{GP}'=\mathscr{GP}(\Gamma',\mathcal{M}')$ is a retract of $\mathscr{GP}$ and so by \cite[Lemma 4.1]{miller:2021} we have
that $\mathscr{GP}'$ is weakly left noetherian. Since $(\alpha,\beta)\notin E$ we have that $\mathscr{GP}'$ is the free
product of $M_{\alpha}$ and $M_{\beta}$.  Lemma~\ref{key13} gives that  $|M_\alpha| = |M_\beta| = 2$.

{\bf Step 2} 
 Let $\gamma\in   V \setminus \{\alpha, \beta\}.$ We now establish $(\alpha, \gamma), (\beta, \gamma) \in E$. 
  Note $M_\alpha$ is not a group by assumption. Let $a \in M_\alpha$   be such that $a$ is not left invertible and choose non-identity $b \in M_\beta$, and $c \in M_\gamma$. Let $K$ be a left ideal of $\mathscr{GP}$ generated by the set $$A = \{[a][b][c]([a][b])^n : n \in \mathbb{N}\}.$$

Suppose, for contradiction, that $(\alpha, \gamma) \notin E$.
Since $(\alpha,\beta),(\alpha, \gamma)\notin E$, it follows that $a \circ b \circ c \circ (a \circ b)^n$ is reduced for all $n\in \mathbb{N}$. As $\mathscr{GP}$ is weakly left noetherian, there exists a finite subset $A'$ of $A$ generating $T$. Let $k$ be the maximum such that $[a][b][c]([a][b])^k \in A'$. Then, there exists some reduced word $v \in X^*$ and $1 \leq m \leq k$ such that
\[
[a][b][c]([a][b])^{k+1} = [v][a][b][c]([a][b])^m.
\]
Again since $(\alpha, \beta), (\alpha, \gamma)\notin E$, we deduce the first block of  $a \circ b \circ c \circ (a \circ b)^m$
in left Foata normal form  is $a$. Let $z$ be a reduced form of $v \circ a$. Since $a$ is not left invertible,
Lemma~\ref{standard multi} gives that  $z \circ b \circ c \circ (a \circ b)^m$ is a reduced form of $v \circ a \circ b \circ c \circ (a \circ b)^m$.  Therefore,
\[
[a][b][c]([a][b])^{k+1} = [z][b][c]([a][b])^m.
\]

Since  both sides of the above equality are in reduced form, we have $$[a][b][c]([a][b])^{k+1-m} = [z][b][c]$$ by Lemma~\ref{block}. Since $(\alpha, \beta), (\alpha, \gamma) \notin E$, the first block of a right Foata normal form of $a \circ b \circ c \circ (a \circ b)^{k+1-m}$ is $b$.  However, on the other hand, $s(c)$ must be in the support of the first block of a right Foata normal form of $z \circ b \circ c$, and thus in that of $a \circ b \circ c \circ (a \circ b)^{k+1-m}$, a contradiction, as $\gamma=\supp(c) \neq \supp(b)=\beta$. Therefore, $(\alpha, \gamma) \in E$ for all $\gamma \in V \setminus \{\alpha, \beta\}$.  A very similar argument, using the same $K$, shows that $(\beta, \gamma) \in E$ for all $\gamma \in V \setminus \{\alpha, \beta\}$. 

  To complete the proof
suppose for contradiction there are infinitely many pairs $(\alpha_i, \beta_i)\not \in E$ with $i\in I$,  $\alpha_i\neq \beta_i$ and either $M_{\alpha_i}$ or $M_{\beta_i}$ is not a group, satisfying conditions obtained in Steps 1 and 2. Without loss of generality, assume $M_{\alpha_i}$ is not a group for all $i\in I$. Then there exist $a_i\in M_{\alpha_i}$ such that $a_i$ is not left invertible for all $i\in I$.  It follows from Step 2 and the comment after Definition ~\ref{d10} that $\alpha_i\neq \alpha_j$ for all $i, j\in I$ with $i\neq j$. Lemma \ref{key4} now gives that  $\{\mathscr{GP}[a_i]: i\in I\}$ is an infinite anti-chain, contradicting the fact that  $\mathscr{GP}$ is weakly left noetherian. Thus the number of pairs satisfying conditions obtained in Steps 1 and 2 is finite.\end{proof}

\begin{Thm}\label{wln}
A graph product $\mathscr{GP}=\mathscr{GP}(\Gamma,\mathcal{M})$  is weakly left noetherian if and only if $\Gamma$ is relatively complete with respect to 
with respect to $\mathcal{M}$,  each vertex monoid is weakly left noetherian and only finitely many vertex monoids are not groups.
\end{Thm}

\begin{proof} If $\mathscr{GP}$ is weakly left noetherian then from Corollary~\ref{cor:retract3}, so is each vertex monoid. Furthermore, by Proposition~\ref{key15}, $\Gamma$ is relatively complete with respect to $\mathcal{M}$. From Proposition~\ref{direct4}
we have that \[
\mathscr{GP} \cong \mathscr{GP}_1 \times \mathscr{GP}_2 \times \mathscr{GP}_3,
\] where the factor $\mathscr{GP}_1$ may be empty. Here $ \mathscr{GP}_2 $ is the restricted direct product of the vertex monoids in $\mathcal{M}_2$. Since 
$ \mathscr{GP}_2 $ is a retract of $\mathscr{GP}$ and the monoids in $\mathcal{M}_2$ are not groups, we have from Proposition~\ref{prop:reducedsimple} that $\mathcal{M}_2$ is finite. 

Conversely, suppose that $\Gamma$ is relatively complete 
with respect to $\mathcal{M}$,  each vertex monoid is weakly left noetherian  and only finitely many vertex monoids are not groups. By Proposition~\ref{direct4},
\[
\mathscr{GP} \cong \mathscr{GP}_1 \times \mathscr{GP}_2 \times \mathscr{GP}_3,
\]
where (if it exists) the factor $\mathscr{GP}_1$  is  the direct product of the free products  $M_{\alpha_i}\ast M_{\beta_i}$ of monoids $M_{\alpha_i}$ and $M_{\beta_i}$ where  $\alpha_i, \beta_i\in V_1$ and $(\alpha_i,\beta_i)\notin E$,  $\mathscr{GP}_2$ is the restricted direct product of monoids in $\mathcal{M}_2$  and $\mathscr{GP}_3$ is the graph product of groups in $\mathcal{M}_3$. 
Lemma~\ref{key13} gives that $\mathscr{GP}_1$ is is weakly left noetherian and also
 $\mathscr{GP}_2$, as $\mathcal{M}_2$ is finite.  Finally, since $\mathscr{GP}_3$ is a group, it is again weakly left noetherian. Another application of Lemma~\ref{key13}  now gives that
$\mathscr{GP}$ is weakly left noetherian. 
\end{proof}

\begin{Rem}\label{rem:wln} We have used Lemma~\ref{key13} (concerning  free products of two monoids) and Proposition~\ref{prop:reducedsimple} (concerning restricted direct products) in the proof of Theorem~\ref{wln}, although their statements may  be read off from that of the theorem. It is also clear from both Lemma~\ref{key13} and Theorem~\ref{wln} that the free product of three or more non-trivial monoids cannot be weakly left noetherian unless they are all groups.\end{Rem}

\section{Reduction of products}\label{sec:red}

For the purposes of Sections~\ref{sec:leftidealH} and \ref{sec:L} we must now make a careful analysis of the reduction of a  product $s\circ a$ of two reduced words $s, a\in X^*$.  Our most general result, Proposition~\ref{prop:reduction}, is needed in full strength in Section~\ref{sec:L}, but a simpler process, as outlined in Corollary~\ref{cor:reduction}, suffices for Section~\ref{sec:leftidealH}.  In what follows  we refer to the 
steps of the process of reduction of a product
 $s\circ a$ that reduce its length as `reduction moves'.  A reduction move that consists of applying  a relation from $R_{\textsf{v}}$, or  first shuffling letters of $s$ and $a$, that is, in applying relations from $R_{\textsf{e}}$, and then applying  a relation from $R_{\textsf{v}}$, we refer to as a  `glueing move';  one that results in subsequently applying a relation from $R_{\textsf{id}}$ 
 we refer to as a `deletion move'.

Note that in the following, for convenience in applications, we do not begin by labelling the elements of $s$; as we reduce a product $s\circ a$ we {\em introduce} convenient labels.

\begin{Prop}\label{prop:reduction} Let $a=a_1\circ\dots\circ a_m\in X^*$ and $s\in X^*$ be reduced.
Suppose that  $s\circ  a$ has been subject to $M$ reduction moves. We have
$[1,M]=I \cup J$ where $I$ labels the $h$ glueing moves and $J$ labels the $k$ deletion moves.
We write $I=\{ i_\ell: 1\leq \ell\leq h\}$ where  $1\leq i_1<\dots <i_h\leq M$,  and
$J=\{ j_o: 1\leq o\leq k\}$ where $1\leq j_1<\dots <j_k\leq M$. If $h=0$, then $I=\emptyset$, if $k=0$ then $J=\emptyset$; notice that $M=h+k$ and 
$M$ is the greater of $i_h$ and $j_k$.

 Then there exists an injection $\theta:[1,M]\rightarrow [1,m]$ and letters $s_{p\theta}$  of $s$ with $\supp(s_{p\theta})=\supp(a_{p\theta})$ for $1 \leq p\leq M$ such that: 
\begin{enumerate}
\item $s$ shuffles to $s^M\circ  s_{i_1\theta}\circ \dots \circ s_{i_h\theta} \circ s_{j_k\theta}\circ \dots\circ s_{j_1\theta}$ where $s_{j_o\theta}a_{j_o\theta}=1$ for all $1\leq o\leq k$
and  $s^M$ is the word $s$ with the letters  $s_{ i\theta}$ deleted, $1\leq i\leq M$;
\item  $a$ shuffles to $a_{j_1\theta}\circ\cdots\circ a_{j_k\theta}\circ a_{i_h\theta}\circ\cdots\circ a_{i_1\theta}\circ a^M$,
where   $a^M$ is the word $a$ with the letters  $a_{ i\theta}$ deleted, $1\leq i\leq M$;
\item $s\circ a\equiv s^M\circ s_{i_1}a_{i_1\theta}\circ\dots\circ s_{i_h}a_{i_h\theta}\circ  a^M$;
\item $ s^M\circ s_{i_1\theta}a_{i_1\theta}\circ\dots\circ s_{i_h\theta}a_{i_h\theta}$ and $s_{i_1\theta}a_{i_1\theta}\circ\dots\circ s_{i_h\theta}a_{i_h\theta}\circ a^M$ are reduced;
\item $s_{i_1\theta}a_{i_1\theta}\circ\dots\circ s_{i_h\theta}a_{i_h\theta} $ is a complete block.
\end{enumerate}
\end{Prop}

\begin{proof} We proceed by  induction on $M$, the number of reduction moves, to show that $h,k, I,J, \theta, s^M$ and $a^M$ exist satisfying conditions (1)-(5).

If $M=1$, then by Remark~\ref{cor:concatenate} there must be some letter $z$ of $s$
such that $z$ shuffles to the end of $s$, a letter $a_{\ell}$ of $a$ such that $a_{\ell}$ shuffles to the start of $a$,
such that $\supp(z)=\supp(a_{\ell})$  and either $za_{\ell}\notin I$  or $za_{\ell}\in I$.
In the first case we have $h=1=i_1$ and $k=0$  and in the second we have  $k=1=j_1$ and $h=0$, so that in either case $M=h+k$ and $M$ is the greater of
$i_h$ and $j_k$. In each case  we put $z=s_{1\theta}$ and 
$\ell=1\theta$ so that
$a_{\ell}=a_{1\theta}$. 
We have that $s$ shuffles to $s^1\circ s_{1\theta}$ where  $s^1$ is the subword of $s$ obtained by removing $s_{1\theta}$,   and $a$ shuffles to $a_{1\theta}\circ a^1$ where $a^1$ is the subword of $a$ obtained by removing $a_{1\theta}$.  If $za_{\ell}=s_{1\theta}a_{1\theta}\notin I$ then, since $s$ and $a$ are reduced, we see that
 $s^1\circ s_{1\theta}a_{1\theta}$ and   $s_{1\theta}a_{1\theta}\circ a^1$  are  reduced; this is a `glueing move'. On the other hand, if $za_{\ell}=s_{1\theta}a_{1\theta}\in  I$ then clearly $s^1$ and $a^1$ are reduced; this is a `deletion move'. In either case we see that
 conditions (1)-(5) hold. 

Suppose now that the result holds for some $K$, and a further reduction move can be applied. By assumption, $h,k,I,J, \theta, s^K$ and $a^K$ exist satisfying conditions (1)-(5).

We now extend the domain of $\theta$ to include $K+1$. Let
$\overline{s}^K :=s^K\circ s_{i_1\theta}a_{i_1\theta}\circ\dots\circ s_{i_h\theta}a_{i_h\theta}$.
By (4) we have that $\overline{s}^K$ and $a^K$ are reduced, but we  are assuming that $\overline{s}^K \circ a^K$  is {\em not} reduced.
By Remark~\ref{cor:concatenate} there is a letter $y$ of $\overline{s}^K$ that shuffles to the end of $\overline{s}^K$ and a letter $a_o$ of $a^K$ that shuffles to the front of $a^K$ such that $\supp(y)=\supp(a_{o})$.
By (4) we have that $\overline{a}^K :=s_{i_1\theta}a_{i_1\theta}\circ\dots\circ s_{i_h\theta}a_{i_h\theta}\circ a^K$ is reduced, and so we must have that  $y$ is a letter from  $s^K$. In view of (1) and (2) we must have that  $s$ shuffles to
\[ s^{K+1}\circ  s_{i_1\theta}\circ \dots \circ s_{i_h\theta} \circ y\circ s_{j_k\theta}\circ \dots\circ s_{j_1\theta}\]
where $s^{K+1}$ is $s^K$ with $y$ deleted and $a$ shuffles to
\[ a_{j_1\theta}\circ\cdots\circ a_{j_k\theta}\circ a_o\circ a_{i_h\theta}\circ\cdots\circ a_{i_1\theta}\circ a^{K+1}\]
where $a^{K+1}$ is $a^K$ with $a_o$ deleted. We   extend  our injection $\theta$ to $[1,K+1]$ by defining $o=(K+1)\theta$. 
If  $ya_{o}\notin I$ then we set $i_{h+1}=K+1$ and put $y=s_{i_{h+1}\theta}$ so that
$ya_{o}=s_{i_{h+1}\theta} a_{(K+1)\theta}\notin I$; this is a `glueing move'. If   $ya_{o}\in I$ then we set $j_{k+1}=K+1$ and put $y=s_{j_{k+1}\theta}$ so that
$ya_{o}=s_{j_{k+1}\theta} a_{(K+1)\theta}\in I$; this is a `deletion move'. In the  case of a glueing move we have increased the cardinality of $I$ by one to obtain $I'$; then  $h+1, k, I',J, \theta, s^{K+1}$ and $a^{K+1}$ are such that conditions (1)-(5) hold. 
In the  case of a deletion  move we have increased the cardinality of $J$ by one to obtain $J'$; then  $h,k+1, I,J',\theta, s^{K+1}$ and $a^{K+1}$  are such that conditions (1)-(5) hold.

The result follows by induction.
\end{proof}

We may follow the process in Proposition~\ref{prop:reduction} until we have obtained a reduced form of $s\circ a$;
in view of Lemma~\ref{product reduction} this will be achieved in
$k$ steps for some $k$ less than or equal to the minimum of the length of $s$ and the length of $a$.  

\begin{defn}\label{defn:ind} Let $a=a_1\circ\cdots\circ a_m\in X^*$ be reduced. A function  $\theta:=\theta(h,k,I,J)$ as in   Proposition~\ref{prop:reduction} is a {\em reduction function} for $a$  if it determines the reduction of a word $s\circ a$ (where $s$ is reduced), to a reduced word.
\end{defn}

Note that there are only finitely many reduction functions associated to  a given reduced word $a$. We further observe that if
$a$ contains no left invertible letters in the first block of a left Foata normal form, that is, that shuffle to the front, then the process described in Proposition~\ref{prop:reduction} will be considerably simplified.
Since we need this simplification to prove the main result of Section~\ref{sec:leftidealH} we state it explicitly.

\begin{Cor}\label{cor:reduction} Let $a=a_1\circ\dots\circ a_m\in X^*$ be a reduced word such that $a$ has no left invertible elements that can be shuffled to the front, and let $s\in X^*$ be reduced. Then $s\circ a$ reduces to a reduced word via $N$ glueing moves.

There is an injection $\theta:[1,N]\rightarrow
[1,m]$
and letters $s_{i}$ of $s$ with $\supp(s_i)=\supp(a_{i\theta})$ for $1\leq i\leq N$
such that:
\begin{enumerate} 
\item $s$ shuffles to $s'\circ s_{1}\circ \cdots \circ s_{N}$ where $s'$ is the word $s$ with the letters  $s_{ i}$ deleted, $1\leq i\leq N$;
\item $a$ shuffles to $ a_{1\theta}\circ \cdots \circ a_{N\theta}\circ a'$ where $a'$ is the word $a$ with the letters  $a_{ i\theta}$ deleted, $1\leq i\leq N$;
\item $s\circ a\equiv s'\circ s_{1}a_{1\theta}\circ \cdots \circ s_{N}a_{N\theta}\circ a'$ where $s'\circ s_{1}a_{1\theta}\circ \cdots \circ s_{N}a_{N\theta}\circ a'$ is reduced;
\item   $s'\circ s_{1}a_{1\theta}\circ \cdots \circ s_{N}a_{N\theta}$ and $s_{1}a_{1\theta}\circ \cdots \circ s_{N}a_{N\theta}\circ a'$ are reduced;
\item $s_{1}a_{1\theta}\circ \cdots \circ s_{N}a_{N\theta}$ is a complete block.
\end{enumerate}
\end{Cor}

Where we are only reducing a single product $s\circ a$, we can also assume, via relabelling $a$ before we begin the reduction, that $\theta$ is the identity in the above result. Unfortunately  in Section~\ref{sec:L} we must keep track of two concurrent reductions $s\circ a$ and $t\circ a$, so the explicit use of $\theta$ will  still be needed.

\bigskip

In Section~\ref{sec:leftidealH} we consider the intersections of principal left ideals.  Clearly
 $\mathscr{GP}[a']\cap \mathscr{GP}[b']$  is non-empty if and only if there exist $u',v'\in X^*$ such that $[u'\circ a']=[v'\circ b']$.  Of course, these products may not be immediately reduced. However, we will be able to reduce our considerations to expressions of the form 
 $u\circ a$  and $v\circ b$, where $[u\circ a]=[v\circ b]$ and $u\circ a$ and $v\circ b$ are reduced. To proceed from that point we will need the following analysis.

\begin{Prop}\label{prop:factorisation}  Suppose that  $u,v,a,b\in X^*$ are such that
$u\circ a, v\circ b$ are reduced and $[u\circ a]=[v\circ b]$.

Then there exist subwords  $a'$ of $a$ and $b'$ of  $b$ and $w\in X^*$  such that $w\circ a'\circ b$ and $w\circ b'\circ a$ are reduced,
\[[u]=[w\circ b'], [v]=[w\circ a'] \mbox{ and }[a'\circ b]=[b'\circ a].\]

Moreover, $a'$ and $b'$ are unique and do not depend upon  $u$ and $v$.
\end{Prop}

\begin{proof} Let $a=a_1\circ\dots\circ a_m$ and $ b=b_1\circ\dots\circ b_n$ be reduced.    Consider the  shuffle from $v\circ b$ to $u\circ a$. Let $b_{i_1} \circ b_{i_2}\circ \cdots \circ b_{i_p}$ be  the longest subword of $b$ that shuffles under this shuffle to a  subword of $u$. Then $u=u_0\circ b_{i_1\sigma}\circ u_1\circ b_{i_2\sigma}\circ \cdots \circ b_{i_p\sigma}\circ u_p$,  for some permutation $\sigma$ of $\{i_1, \cdots, i_p\}$, and where
$u_i$ is a (possibly empty) subword of $u$
for $0\leq i\leq p$. Similarly, considering the shuffle in the opposite direction, let
$a_{j_1} \circ a_{j_2}\circ \cdots \circ a_{j_q}$ be  the longest subword of $a$ that shuffles under this shuffle  to a subword
of $v$. Then
$v=v_0\circ a_{j_1\delta}\circ v_1\circ a_{j_2\delta}\circ \cdots \circ a_{j_q\delta}\circ v_q$,  for some permutation $\delta$ of $\{j_1, \cdots, j_q\}$, and where
$v_i$ is a (possibly empty) subword of $v$
for $0\leq i\leq q$. Notice that $m=q+s$ and $n=p+s$ where $s$ is the number of letters of $a$ and $b$ that shuffle to each other, so that $p+m=q+n$.

Suppose that $u_{{\ell}}=x\circ c$ where $x\in X, c\in X^*$. Then there must be a letter of $v$ that shuffles to $x$ under the shuffle from $v\circ b$ to $u\circ a$. Since in $v\circ b$, every letter of $v$ precedes every letter of $b$,
 it follows from Proposition~\ref{easy observation} that $b_{i_\ell\sigma}\circ x$ shuffles to $x\circ b_{i_\ell\sigma}$.
   Continuing in this manner
 we obtain that $u$  shuffles to $w_u\circ b''$ where $w_u=u_0\circ u_1\circ\cdots \circ u_p$ and
 $b''=b_{i_1\sigma} \circ b_{i_2\sigma}\circ \cdots \circ b_{i_p\sigma}$. It follows from
 Proposition~\ref{easy observation} that $b''$ shuffles to $b'$ where $b'=b_{i_1} \circ b_{i_2}\circ \cdots \circ b_{i_p}$
 and so $u$ shuffles to $w_u\circ b'$.
Dually, $v$ is shuffle equivalent to $w_v\circ a'$ where $w_v=v_0\circ v_1\circ\cdots \circ v_q$ and  $a'=a_{j_1} \circ a_{j_2}\circ \cdots \circ a_{j_p}$.

We now have that $w_u\circ b'\circ a$ and $ w_v\circ a'\circ b$ are shuffle equivalent  to $u\circ a$ and $v\circ b$; in particular, they are reduced.   An analysis of the four shuffling processes, given that the letters of $w_u$ and $w_v$ are precisely those appearing in neither $a$ nor $b$, yields that $w_v$ shuffles to $w_u$.
We call upon Lemma~\ref{block} to deduce that
 $[w_u]=[w_v]:=[w]$ and $[a'\circ b]=[b'\circ a]$; we have already shown that $[u]=[w\circ b']$ and $[v]=[w\circ a']$.

We now show that $a'$ and $b'$ are the  unique subwords of $a$ and $b$ such that $[a'\circ b]=[b'\circ a]$, and hence that they do not depend upon $u$ and $v$.
We proceed by induction on $|a|+|b|$.  If $|a|+|b|=0$,  the result is clear.

Assume  that $|a|+|b|>0$ and the result is true for all
reduced words $\bar{a}$ and $\bar{b}$ with $|\bar{a}|+|\bar{b}|<|a|+|b|$ and such that $[ \bar{b}'\circ \bar{a}]=[\bar{a}'\circ \bar{b}]$ for some subwords $\bar{a}'$ of $\bar{a}$ and $\bar{b}'$ of $\bar{b}$. Without loss of generality, assume that $|b|>0$.
Consider subwords $a'$ and $b'$  of $a$ and $b$  respectively such that   $[a'\circ b]=[b'\circ a]$.

Suppose first that in the shuffle from
$a'\circ b$ to $b'\circ a$ the last letter of $b$ shuffles to a letter $c$ of $a$.
We note that this happens {\em if and only if} $\alpha=s(c)\in \supp(a)$, and does not depend upon the choice of $a'$ and $b'$. In this case,
$c$ must be the last letter of $a$ to have support $\alpha$. Lemma~\ref{block} gives that 
$[a'\circ \bar{b}]=[b'\circ \bar{a}]$, where $\bar{b}$ is $b$ with the last letter, $c$, deleted, and $\bar{a}$ is $a$ with $c$ deleted.
Notice that $a'$ is a subword of $\bar{a}$ and $b'$ is a subword of $\bar{b}$ and $|\bar{a}|+|\bar{b}|<|a|+|b|$. We obtain the uniqueness of
$a'$ and $b'$ from the induction hypothesis.

Finally, suppose  that in the shuffle from
$a'\circ b$ to $b'\circ a$ the last letter of $b$ shuffles to a letter $c$ of $b'$,  which clearly must be the last letter of $b'$.
It follows that $(\supp(c),\supp(d))\in E$ for all letters $d$ of $a$. Again, let $\bar{b}$ be the word $b$ with the last letter, $c$, deleted,
and let $\overline{b}'$ be $b'$ with the last letter, $c$, deleted. Then $a'\circ \bar{b}$ is reduced and again calling upon 
Lemma~\ref{block} we
have that $[a'\circ \bar{b}]=[\overline{b}'\circ a]$ (and $\overline{b}'\circ a$ is also reduced). We have $|a|+|\bar{b}|<|a|+|b|$
and $\overline{b}'$ is a subword of $\bar{b}$, so that $a'$ and $\overline{b}'$  are uniquely determined by $a$ and $\bar{b}$. Consequently,
$a'$ and $b'$  are uniquely determined by $a$ and $b$.
\end{proof}

In the above discussion,  given that there are {\em some} $u,v$ with $[u\circ a]=[v\circ b]$, the emerging equality
$[a'\circ b]=[b'\circ a]$ is independent of the values of the letters of $a$ and $b$ and depends only on their support and whether corresponding letters match. To this end we introduce the following.

\begin{defn}\label{defn:supportskeleton}
Let $a=a_1\circ \cdots \circ a_n$ and $b=b_1\circ \cdots \circ b_n \in X^*$. We say that $a$ and $b$ have the {\it same support skeleton} if $\supp(a_i)=\supp(b_i)$ for all $1\leq i\leq n$.
\end{defn}

Lemma ~\ref{one reduction-1} follows from the fact that for any words $w,w'\in X^*$ with the same support skeleton, any sequence of shuffles applied to $w$ can be applied to $w'$, and vice versa.

\begin{Lem}\label{one reduction-1}
Let $a, b, u$ and $ v$ be as  in Proposition \ref{prop:factorisation}. Let $c=c_1\circ \cdots \circ c_m$ and $ d=d_1\circ \cdots \circ d_n$
have the same support skeleton as $a$ and $b$, respectively, and be such that
there are words $p,q\in X^*$ such that  $p\circ c$ and $ q\circ d$ are reduced and $[p\circ c]=[q\circ d]$.

 Let $d'=d_{i_1}\circ \cdots \circ d_{i_p}$ and $c'=c_{j_1}\circ \cdots \circ c_{j_q}$, where $i_1, \cdots, i_p, j_1\circ \cdots j_q$ are the subscripts chosen in Proposition~\ref{prop:factorisation}.  Then there exists $g\in X^*$ such that $g\circ d'\circ c$ and $g\circ  c'\circ d$ are reduced, \[[p]=[g\circ d'], [q]=[g\circ c'] \mbox{ and }[d'\circ c]=[c'\circ d].\]
\end{Lem}

We remark that in Lemma~\ref{one reduction-1} it is certainly not true that for any words $c$ and $d$ with the same support skeleton as $a$ and $b$ we have
that such a $p$ and $q$ exist. We now consider the case where such elements do exist.

\begin{Lem}\label{lem:reduction-3} Let $a=a_1\circ\cdots \circ a_m$ 
and  $b=b_1\circ\cdots \circ b_n$ be reduced. Suppose there are  subwords
\[a'=a_{j_1}\circ\cdots\circ a_{j_l}\mbox{ and }b'=b_{i_1}\circ\cdots\circ b_{i_k}\]
of $a$ and 
$b$ respectively such that
\[[b'\circ a]=[ a'\circ b]\]
where  $b'\circ a$ and $a'\circ b$ are both  reduced.  Let $a''=a_{u_1}\circ \cdots \circ a_{u_{m-l}}$ and
$b''=b_{v_1}\circ\cdots \circ b_{v_{n-k}}$ be the subwords of $a$ and $b$ obtained by deleting $a'$ and $b'$, respectively. Note that $m-l=n-k=w$; let
$\theta:[1,w]\rightarrow [1,w]$ be the bijection induced by the shuffle from $b'\circ a$ to $a'\circ b$ such that $a_{u_r}$ shuffles to $b_{v_{r\theta}}$.

Now let $c=c_1\circ\cdots \circ c_m$ 
and  $d=d_1\circ\cdots \circ d_n$ be reduced words with the same support skeletons as $a$ and $b$, respectively such that $c_{u_r}=d_{v_{r\theta}}$ for $r\in [1,w]$. Then
\[[c'\circ d]= [d'\circ c]\]
where 
\[c'=c_{j_1}\circ\cdots\circ c_{j_l}\mbox{ 
and }
d'=d_{i_1}\circ\cdots\circ d_{i_k}.\]
\end{Lem}

\begin{proof} Since $c$ and $d$ have the same support skeleton as $a$ and $b$ respectively,
the same sequence of moves applied to take $b'\circ a$ to $a'\circ b$ will take
$d'\circ c$ to $c'\circ d''$ where $d''=d''_1\circ\cdots\circ d_n''$ is a word of length $n$ such that 
$d''_{i_r}=d_{i_r}$. For $j\notin \{ i_1,\cdots, i_k\}$, that is,
$j\in\{ v_1,\cdots, v_w\}$ we have $j=v_{s\theta}$ for some $s\in [1,w]$ and
$c_{u_s}=d_{v_{s\theta}}$ shuffles to $d''_{v_{s\theta}}$ so that  $d''_{v_{s\theta}}=d_{v_{s\theta}}$ and so $d''=d$.
\end{proof}

\begin{Rem}\label{one reduction-2}
Let $s, t, u, v, a, b$ be reduced such that both $a$ and $b$ have no left invertible elements that can be shuffled to the front. Suppose that $$[s\circ a]=[t\circ b]\mbox{~and~}[u\circ a]=[v\circ b].$$ Corollary~\ref{cor:reduction} gives that  $s\circ a$ and $u\circ a$ reduce to $s'\circ c$ and $u'\circ e$, respectively, such that $c$ and $e$ have the same support skeleton as $a$. Similarly,  $t\circ b$ and $v\circ b$ reduce to $t'\circ d$ and $v'\circ f$, respectively, such that $d$ and $f$ have the same support skeleton as $b$. Therefore, we have $$[s'\circ c]=[t'\circ d]\mbox{ and }[u'\circ e]=[v'\circ f].$$ By Proposition \ref{prop:factorisation}, we have $[d'\circ c]=[c'\circ d]$ where $c'=c_{j_1}\circ \cdots \circ c_{j_l}$ is a subword of $c$ and $d'=d_{i_1}\circ \cdots \circ d_{i_k}$ is a subword of $d$. Let $e'=e_{j_1}\circ \cdots \circ e_{i_l}$ be a subword of $e$ and $f'=f_{i_1}\circ \cdots \circ f_{i_k}$ be a subword of $f$. Then $[f'\circ e]=[e'\circ f]$ by Lemma \ref{one reduction-1}.
\end{Rem}

We need one more refinement of Proposition~\ref{prop:factorisation} in both Sections~\ref{sec:leftidealH} and ~\ref{sec:L}. 

\begin{Lem}\label{lem:thedoubleshuffle}   Let $a=a_1\circ\cdots \circ a_m$ 
and  $b=b_1\circ\cdots \circ b_n$ be reduced words. Suppose there are  subwords
\[a'=a_{j_1}\circ\cdots\circ a_{j_l}\mbox{ and }b'=b_{i_1}\circ\cdots\circ b_{i_k}\]
of $a$ and 
$b$ respectively such that
\[[b'\circ a]=[ a'\circ b]\]
where  $b'\circ a$ and $a'\circ b$ are both  reduced. In addition, suppose that 
\[a= a^{\ell}\circ a^r\mbox{ and }b= b{^\ell}\circ b^r\]
where $a^\ell$ and $b^\ell$ are complete blocks (so (left) factors of the first block of $a$ and $b$ in left Foata normal form, respectively). 
In addition, we let $a'=a^{(\ell)}\circ a^{(r)}$, where $a^{(\ell)}$ is a subword of $a^\ell$ and $a^{(r)}$ is a subword of $a^r$;
we define $b^{(\ell)}$ and $b^{(r)}$ correspondingly. Let $a^\ell=a_1\circ\cdots\circ  a_M$ and let $b^\ell=b_1\circ \cdots \circ b_N$.   Considering the shuffle $b'\circ a\equiv a'\circ b$ let 
\[\begin{array}{rclrcl}
I_{a}&=& \{ i\in [1,M]: a_i\mbox{ shuffles to a letter of  }b^\ell\}& 
I_b&=& \{ j\in [1,N]: b_j\mbox{ shuffles to a letter of }a^\ell\}\\
J_a&=&[1,M]\cap \{ j_1,\cdots, j_l\}& J_b&=&[1,N]\cap \{ i_1,\cdots ,i_k\}\\
K_a&=&[1,M]\setminus (I_{a}\cup J_a)&K_b&=&[1,N]\setminus ( I_b\cup J_b).
\end{array}\]
By construction,  $I_a,J_a$ and $K_a$ are mutually disjoint, as are $I_b,J_b$ and $K_b$, and there is a bijection $\sigma:I_{a}\rightarrow I_b$ such that $a_i$ shuffles to $b_{i\sigma}$. Further,  
\[a'\circ b\equiv a^{(\ell)}\circ b^{(\ell)}\circ w_{a,b}\circ w \equiv b'\circ a\]
where $a^{(\ell)}\circ b^{(\ell)}\circ w_{a,b}$ is a complete block, $a^{(\ell)}$ is the product of the letters $a_i, i\in J_a$, $b^{(\ell)}$ is a product of the letters
$b_i, i\in J_b$, $w_{a,b}$ is a product of the letters $a_i=b_{i\sigma}$ where $i\in I_a$ and $w$ consists of letters from $a^r$ and $b^r$.
\end{Lem}

\begin{proof} By definition, the letters $a_i,i\in J_a$ and  $b_j, j\in J_b$ shuffle to the front of $a'\circ b$ and $b'\circ a$ and as
$a'\circ b\equiv b'\circ a$ they are in the
 first block of $a'\circ b$ in left Foata normal form. Consider now a letter $a_i$ where $i\in I_{a}$ and suppose that $a_i$ shuffles to $b_{i\sigma}$ where $i\sigma\in I_b$. The letter $b_{i\sigma}$  precedes or shuffles with any letter in the subword
 $b_{i_1}\circ \cdots \circ b_{i_k}$ of $b$  and hence also  in  $a'\circ b$. But in the shuffle equivalent word
 $b'\circ a$, each such $b_{i_j}$ precedes $a_i=b_{i\sigma}$. It follows that $a_i$ shuffles to the front of the subword  $b_{i_1}\circ \cdots \circ b_{i_k}$ in
 $b'\circ a$ as required. 
 
Consider now the remaining letters  of $a'\circ b\equiv a^{(\ell)}\circ a^{(r)}\circ b^\ell\circ b^r$. Letters of $a^{(r)}$ are letters from $a^r$; letters
 from $b^{\ell}$ either shuffle into $b^{(\ell)}=w_b$ or to $a^\ell$ and hence into $w_{a,b}$ or into $a^r$. This concludes the proof.
\end{proof}

\section{Left ideal Howson}\label{sec:leftidealH}

Recall that a monoid $M$ is said to be {\it left ideal Howson} if the intersection of any two finitely generated left ideals, or, equivalently, any two principal left ideals, is again finitely generated. This term was introduced in honour of the author of \cite{howson:1954},
 who showed that the intersection of finitely generated subgroups of free groups is finitely generated.
By very definition, a weakly left noetherian monoid is left ideal Howson. The converse is not true, as may be seen by considering
 the natural numbers under the operation of maximum, with an identity adjoined. The left ideal Howson property has been extensively considered for other classes of semigroups, (see \cite{jones:1989,jones:2016,lawson:2019, silva:2016}). Further, it was shown in \cite{carson:2021} that the class of left ideal Howson monoids is closed under taking both free products and direct products.
The aim of this section is to show that the class of left ideal Howson monoids is closed under taking graph products.

\begin{Thm}\label{thm:howson}
A graph product of monoids is left ideal Howson if and only if each vertex monoid is left ideal Howson.
\end{Thm}
\begin{proof} Let $\mathscr{GP}=\mathscr{GP}(\Gamma,\mathcal{M})$.
Suppose that each vertex monoid  is left ideal Howson. As remarked, to show that  $\mathscr{GP}$ is left ideal Howson, it is enough to show that the intersection of two principal left ideals is finitely generated. To this end, consider $\mathscr{GP}[a]$ and $\mathscr{GP}[b]$
where $a=a_1\circ \cdots \circ a_m$ and $ b= b_1\circ\cdots \circ b_n\in X^*$ are reduced.  By Lemma \ref{key5}, we can assume  that the left Foata normal forms of $a$ and $b$ are standard, that is, that there are no left invertible letters in $a$ (respectively $b$) that can be shuffled to the front of $a$ (respectively $b$).

We  show that $\mathscr{I}:=\mathscr{GP}[a]\cap \mathscr{GP}[b]$ is finitely generated. By convention, this is true if
$\mathscr{I}=\emptyset$. Suppose therefore that $\mathscr{I}\neq \emptyset $ and  let
$[z]\in \mathscr{I}$. Let ${s}$ and  ${t}$ be reduced words such that $[z]=[{s}][a]=[{t}][b].$

By
Corollary~\ref{cor:reduction}  and the subsequent comment,  either ${s}\circ a$ is reduced or reduces in $M$ glueing moves, where $M\leq m$, giving rise to a (re)labelling of
the words  ${s}$ and $a$ such that:
\begin{enumerate} \item $\supp(s_i)=\supp(a_{i})$ for $1\leq i\leq M$;
\item ${s}$ shuffles to $s'\circ s_{1}\circ \cdots \circ s_{M}$;
\item $a$ shuffles to  $ a_{1}\circ \cdots \circ a_{M}\circ a^M$;
\item ${s}\circ a$ reduces to the reduced word $s'\circ s_{1}a_{1}\circ \cdots \circ s_{M}a_{M}\circ a^M$;
\item $s_{1}a_{1}\circ \cdots \circ s_{M}a_{M}$ is a complete block.
\end{enumerate}

Similarly, we have a word  $t$ such that $t\circ b$ is reduced or  reduces in $N$ glueing moves, where $N$ is the length of the first block of $b$ in left Foata normal form, giving  rise  to a (re)labelling of
the words  $t$ and $b$ such that:

\begin{enumerate} \item $\supp(t_i)=\supp(b_i)$ for $1\leq i\leq N$;
\item $t$ shuffles to $t'\circ t_{1}\circ \cdots \circ t_{N}$;
\item $b$ shuffles to $ b_{1}\circ \cdots \circ b_{N} \circ b^N$;
\item $t\circ b$ reduces to the reduced word $t'\circ t_{1}b_{1}\circ \cdots \circ t_{N}b_{N}\circ b^N$;
\item $ t_{1}b_{1}\circ \cdots \circ t_{N}b_{N}$ is a complete block.
\end{enumerate}

We now have
\[[s\circ a]=[s'\circ c]=[t'\circ d]=[t\circ b]\] where
\[c =s_{1}a_{1}\circ \cdots \circ s_{M}a_{M}\circ a^M
\mbox{ and }
d= t_{1}b_{1}\circ \cdots \circ t_{N}b_{N}\circ b^N.\]
Note that $c$ and $d$ have the same support skeletons as $a$ and $b$, respectively.
Let $c'$ and $d'$ be the unique subwords of $c$ and $d$ guaranteed by
Proposition~\ref{prop:factorisation} such that
\[[s']=[u\color{black}\circ d'], [t']=[u\color{black}\circ c']\mbox{ and }[d'\circ c]=[c'\circ d]\]
for some  $u\in X^*$. \color{black}

It follows that $[s\circ a]=[t\circ b]\in \mathscr{GP}[v]$ where $[v]=[d'\circ c]=[c'\circ d]$. Now observe that
\[[v]=[d' \circ s_{1}\circ \cdots \circ s_{M}\circ a]=[c' \circ t_{1}\circ \cdots \circ t_{N}\circ b], \]
so that $[v]\in \mathscr{I}$.

We return  to the reduced and equivalent words  $d'\circ c$ and $c'\circ d$. 
We will employ Lemmas~\ref{lem:reduction-3} and \ref{lem:thedoubleshuffle} with $c$ and $d$ taking the role of $a$ and $b$. To that end, let 
\[c^{\ell}=s_{1}a_{1}\circ \cdots \circ s_{M}a_M,\, c^r=a^M, \, d^{\ell}=t_{1}b_{1}\circ \cdots \circ t_{N}b_N\mbox{ and }d^r=b^{N}.\]
We let $j_1,\cdots, j_l, i_1,\cdots, i_k,\color{black} I_c, I_d, J_c,J_d$, $K_c$ and $K_d$ be defined as in Lemma~\ref{lem:thedoubleshuffle}.
Then 
\[d'\circ c\equiv c^{(\ell)}\circ d^{(\ell)}\circ w_{c,d}\circ w\equiv c'\circ d\]
where $w$ consists of letters of $a^M$ and $b^N$.

Let $a^{(\ell)}$ be the reduced word with the same support as $c^{(\ell)}$, but replacing each $s_ia_i$ by $a_i$; let $b^{(\ell)}$ be defined similarly, and let
$w'=a^{(\ell)}\circ b^{(\ell)}\circ w$. Note that $w'$ consists only of letters of $a$ and $b$.  
Let $\sigma:I_c\rightarrow I_d$ be defined as in Lemma~\ref{lem:thedoubleshuffle} so that  $s_oa_o=t_{o\sigma}b_{o\sigma}$
 and $w_{c,d}$ is the product of letters $s_oa_o$ where $o\in I_c$. \color{black} We now call upon the fact that all the vertex monoids are left ideal Howson, and we let $X_o$
 be a finite set of
generators of the left ideal  $M_oa_o\cap M_{o}b_{o\sigma}$ for all $o\in I_{c}:=\{ o_1,\cdots, o_g\color{black}\}$, where $o=\supp(a_o)$. We claim that $X$ is a finite set of generators
for $\mathscr{I}$, where
\[X=\{ [x_{o_1}\circ\cdots \circ x_{o_g\color{black}}\circ w']:   x_{o_i}\in X_{o_i}, 1\leq i\leq g\color{black}\}.\]

First, we show that $\mathscr{GP}X\subseteq \mathscr{I}$. Let $x=x_{o_1}\circ\cdots \circ x_{o_g\color{black}}\circ w'$ so that $[x]\in X$. For $1\leq i\leq g$ choose $p_{o_i}, q_{o_i\sigma}\in M_{o_i}$ such that
$x_{o_i}=p_{o_i}a_{o_i}=q_{o_i\sigma}b_{o_i\sigma}$.

We form a word $s_1'\circ\cdots \circ s_M'$
with the same support skeleton as $s_1\circ\cdots \circ s_M$ as follows.
If  $o\in I_{c}$ then put $s_o'=p_o$; if $o\in J_c$ we put $s_o'=I_o$, where $s_o\in M_o$; if $o\in K_c$  let $s_o'=s_o$.   Similarly,  we form a word $t_1'\circ\cdots \circ t_N'$
with the same support skeleton as $t_1\circ\cdots \circ t_N$ as follows.
If  $o\sigma\in I_d$ then put $t_{o\sigma}'=q_{o\sigma}$; if $o\in J_d$ we put $t_o'=I_o$, where $t_o\in M_o$; if $o\in K_d$  let $t_o'=t_o$. 
We now define 
\[\bar{c}=s_1'\circ\cdots \circ s_M'\circ a= s_{1}'a_{1}\circ \cdots \circ s_{M}'a_{M}\circ a^M
\]and\[\bar{d}=t_1'\circ\cdots \circ t_N'\circ b=
 t_{1}'b_{1}\circ \cdots \circ t_{N}'b_{N}\circ b^N.\]
 We  call upon Lemma~\ref{lem:reduction-3} to obtain subwords $\bar{c}'=\bar{c}_{j_1}\circ\cdots\circ \bar{c}_{j_l}$ and 
 $\bar{d}'=\bar{d}_{i_1}\circ\cdots\circ \bar{d}_{i_k}$ of $\bar{c}$ and $\bar{d}$ respectively, so that observing that 
 $\bar{c}^{(\ell)}=a^{(\ell)}$ and $\bar{d}^{(\ell)}=b^{(\ell)}$ we have
 \[\bar{d}'\circ \bar{c}\equiv a^{(\ell)}\circ b^{(\ell)}\circ w_{\bar{c},\bar{d}}\circ w\equiv \bar{c}'\circ \bar{d}\equiv w_{\bar{c},\bar{d}} \circ w'=[x]\]
 where $w_{\bar{c},\bar{d}}=x_{o_1}\circ\cdots\circ x_{o_g\color{black}}$. 
This gives that
\[[x]=[ \bar{d}'\circ s_1'\circ\cdots \circ s_M'\circ a]=[\bar{c}'\circ t_1'\circ\cdots \circ t_N'\circ b]\in \mathscr{I}\]
 and, consequently, $\mathscr{GP}X\subseteq\mathscr{I}$. \color{black}

Now we show that $\mathscr{GP}X\supseteq \mathscr{I}$. To this end, let us return to an arbitrary $[z]\in \mathscr{I}$. Then
$[z]=[s\circ a]=[t\circ b]$ where $s$ and $t$ are obtained as at the start of the proof. Note that it was the existence of {\em any} such $[z]$
that determined the support skeleton of $d'\circ c$ and $c'\circ d$; the support skeleton does not depend upon the choice of $z$ and by  Remark \ref{one reduction-2} is unique.
For $o_i\in I_{c}$ we have that $s_{o_i}a_{o_i}=t_{o_i\sigma}b_{o_i\sigma}$ so that $s_{o_i}a_{o_i}=k_{o_i}x_{o_i}$
for some $x_{o_i}\in X_{o_i}$.  Put $x_{o_i}=p_{o_i}a_{o_i}$ as above  for $1\leq i\leq g\color{black}$ and  let
$x=x_{o_1}\circ\cdots \circ x_{o_g\color{black}}\circ w'$ so that $[	x]\in X$. Now list the elements of $J_c$ and $J_d$ respectively as
\[J_c=\{ e_1,\cdots ,e_{h\color{black}}\}\mbox{ and }J_d=\{ f_1,\cdots, f_r\}.\]
It follows that
\[[s\circ a]=[u\color{black}\circ (k_{o_1}\circ\cdots\circ k_{o_g\color{black}} \circ s_{e_1}\circ \cdots \circ s_{e_{h\color{black}}}\circ t_{f_1}\circ\cdots \circ t_{f_r})][x]\in \mathscr{GP}X.\]
This completes the proof that $X$ generates $\mathscr{I}$ and hence the proof that $\mathscr{GP}$ is left ideal Howson.

Conversely, if $\mathscr{GP}$ is left ideal Howson, then so is each vertex monoid, since the class of left ideal Howson monoids is closed under retract \cite{dasar:2024}.
\end{proof}

We recall  that a monoid $M$  is an {\em LCM} monoid if for any $a,b\in M$ we have $Ma\cap Mb=\emptyset$ or $Ma\cap Mb=Mc$ for some $c\in M$. Examining the choice of elements of $X$ in the proof of Theorem~\ref{thm:howson} gives one direction of the  following, the other due to an easy consideration of retracts;  it appears in \cite{fountain:2009} in the case where the vertex monoids are right cancellative  LCM monoids, and hence embed into their inverse hulls.

\begin{Cor}
A graph product of monoids is an  LCM monoid if and only if each vertex monoid is an LCM monoid.
\end{Cor}

Some texts (e.g. \cite{exel:2018}) use similar terminology to LCM for the condition where $Ma\cap Mb=\emptyset$ is not permitted; we say here that a monoid $M$ is a {\em strong LCM} monoid if for any $a,b\in M$ we have  $Ma\cap Mb=Mc$ for some $c\in M$.  If a graph product $\mathscr{GP}$ is a strong LCM monoid, then again considering retracts, each vertex monoid has the same property. On the other hand, if $M_{\alpha}$ and $M_{\beta}$ are non-group vertex monoids with the strong LCM property,
and $(\alpha,\beta)\notin E$, then it is easy to see that $\mathscr{GP}[a]\cap \mathscr{GP}[b]=\emptyset$ for any non-units $a\in M_{\alpha},b\in M_{\beta}$.

\begin{Cor}\label{cor:rih} cf. \cite{carson:2021} A free product or a restricted direct product of monoids $\{ M_i:i\in I\}$ is left ideal Howson if and only if so is each $M_i, i\in I$.
\end{Cor}

\section{Finitely left equated}\label{sec:L}

A {\em left congruence} on a monoid $M$ is an equivalence relation closed under multiplication by elements of $M$ on the left. We will denote the smallest left congruence containing a subset $X$ of $M\times M$ by $\lambda_X$.  For any $a\in M$ it is easily seen that
the left annihilator  
$$\mathbf{l}(a):=\{(s, t)\in M\times M: sa=ta\}$$ 
is a left congruence.  We say that $M$ is {\em finitely left equated} if  
$\mathbf{l}(a)$  is finitely generated for all $a\in M$.  Notice that this property is much weaker than being left noetherian. For instance, it was proved in \cite[Corollary 4.8]{dasar:2024} that every regular monoid is finitely left equated. Regular monoids need not even be weakly  left noetherian, as the example of the natural numbers under maximum, with an identity adjoined, demonstrates. The aim of this section is to prove the following theorem.

\begin{Thm}\label{thm:aagh} A graph product of monoids $\mathscr{GP}=\mathscr{GP}(\Gamma,\mathcal{M})$ is finitely left equated if and only if so is each vertex monoid.
\end{Thm}
\begin{proof}  The forward direction follows from 
\cite{dasar:2024}, which tells us the fact that the class of finitely left equated monoids is closed under retract.

We prove the reverse direction   by  examining the left annihilators of elements of $\mathscr{GP}$.
For the remainder of this  section we consider  $\mathscr{GP}=\mathscr{GP}(\Gamma,\mathcal{M})$, where each vertex monoid is finitely left equated. We choose and fix
 $[a]\in \mathscr{GP}$ where $a=a_1\circ\dots\circ a_m$ is  reduced. Our aim is to find a finite set of generators for $\mathbf{l}([a])$. Unlike the situation in Section ~\ref{sec:leftidealH}, we cannot assume that $a$ has no left invertible elements in its first block in left Foata normal form, which adds to the complications we encounter.

For each left invertible letter $a_i$ of $a$, we choose and fix a left inverse $a_i'$ so that $a_i'a_i=1_{\mathbf{s}(a_i)}$  in $M_{\mathbf{s}(a_i)}$ and 
$[a_i'][a_i]=[a_i'\circ a_i]=[\epsilon]=1$ in $\mathscr{GP}$. Note that we may not need to call upon all such elements.  For each $1\leq i\leq m$ we can pick a  set of generators $X_i$  of the left annihilator $\mathbf{l}(a_i)$ in the vertex monoid $M_{\mathbf{s}(a_i)}$, that is,  $\lambda_{X_i}=\mathbf{l}(a_i)$. Again we note that we may not need to call upon all sets $X_i$. 

Let $\theta=\theta(h,k,I,J)$ be a reduction function for $a$ with labelling $I=\{ i_\ell: 1\leq \ell\leq h\}$ (where $I$ labels the glueing moves and $1\leq i_1<\dots <i_h\leq M$) and $J=\{ j_o: 1\leq o\leq k\}$
(where $J$ labels the deletion moves and $1\leq j_1<\dots <j_k\leq M$).
We  let
\[Y^{\theta}=\{ ([p\circ a_{j_{l-1}\theta}'\circ\ \dots \circ a_{j_1\theta}'], [q\circ a_{j_{l-1}\theta}'\circ\ \dots \circ a_{j_1\theta}']): (p, q)\in X_{j_l\theta}, 1\leq l\leq k\}.\]

\begin{Lem}\label{lem:craigtrick} With notation as above we have $Y^{\theta}\subseteq \mathbf{l}([a])$. Further, 
for any  left inverses $s_{j_l}$ of $a_{j_l}$, where $1\leq l\leq k$,  we have
\[ [s_{j_k\theta}\circ \dots\circ s_{j_1\theta}]\, \, \lambda_{Y^{\theta}}\, \, [a'_{j_k\theta}\circ \dots\circ a'_{j_1\theta}].\]
\end{Lem}

\begin{proof} Since $\theta$ is a reduction function for $a$, we have $$[a]=[a_{j_1\theta}\circ\cdots\circ a_{j_k\theta}\circ a_{i_h\theta}\circ\cdots\circ a_{i_1\theta}\circ a^M].$$
Let $$([p\circ a_{j_{l-1}\theta}'\circ\ \dots \circ a_{j_1\theta}'], [q\circ a_{j_{l-1}\theta}'\circ\ \dots \circ a_{j_1\theta}'])\in Y^\theta.$$ Then
\[\begin{array}{rcl}p\circ a_{j_{l-1}\theta}'\circ \cdots \circ a_{j_1\theta}'\circ a
&\equiv &(pa_{j_l\theta})\circ a_{j_{l+1}\theta}\circ\cdots\circ a_{j_k\theta}\circ a_{i_h\theta}\circ\cdots\circ a_{i_1\theta}\circ a^M\\
&\equiv & (qa_{j_l\theta})\circ a_{j_{l+1}\theta}\circ\cdots\circ a_{j_k\theta}\circ a_{i_h\theta}\circ\cdots\circ a_{i_1\theta}\circ a^M\\
&\equiv &q\circ a_{j_{l-1}\theta}'\circ \cdots \circ a_{j_1\theta}'\circ a\end{array}\]
using the fact that  $pa_{j_l\theta}=qa_{j_l\theta}$.  Hence $Y^\theta\subseteq  \mathbf{l}([a])$.

We show by finite induction on $l$ with $1\leq l\leq k$ that  $[s_{j_l\theta}\circ \dots\circ s_{j_1\theta}]\, \, \lambda_{Y^\theta} \, \, [  a'_{j_l\theta}\circ \dots\circ a'_{j_1\theta}]$. Notice that for any such $l$ we have that $s_{j_l\theta}a_{j_l\theta}=a'_{j_l\theta}a_{j_l\theta}=1_{\mathbf{s}(a_{j_l\theta})}$ so that
$(s_{j_l\theta},a'_{j_l\theta})\in  \mathbf{l}(a_{j_l\theta})$.

Suppose that  $l=1$. Then, since $X_{j_1\theta}\subseteq Y^\theta$ and 
$(s_{j_l\theta},a'_{j_l\theta})\in  \mathbf{l}(a_{j_1\theta})$, we have
that $s_{j_1\theta}\, \lambda_{X_{j_1\theta}} \, a'_{j_1\theta}$ in $M_{\alpha_1}$, where $\supp(s_{j_1\theta})=\alpha_1$, and so
$[s_{j_1\theta}]\, \lambda_{Y^\theta} \, [a'_{j_1\theta}]$ in $\mathscr{GP}$.

Suppose for induction that
\[[s_{j_{l-1}\theta}\circ \dots\circ s_{j_1\theta}]\, \, \lambda_{Y^{\theta}} \, \, [  a'_{j_{l-1}\theta}\circ \dots\circ a'_{j_1\theta}].\]
There exists an  $X_{j_l\theta}$-sequence
\[ s_{j_l\theta}=c_1p_1,\, c_1q_1=c_2p_2,\cdots, c_rq_r=a'_{j_l\theta}\]
in $M_{\alpha_l}$, where $\supp( s_{j_l\theta})=M_{\alpha_l}$. It follows that there exists a $Y^\theta$-sequence
\[[s_{j_l\theta}\circ x]=[c_1][p_1\circ x],\, [c_1][q_1\circ x]=[c_2][p_2\circ x],\cdots, [c_r][q_r\circ x]=[a'_{j_l\theta}\circ x]\]
where $[x]=[a'_{j_{l-1}\theta}\circ \cdots \circ a'_{j_{1}\theta}]$ in $\mathscr{GP}$.
The inductive hypothesis now yields
\[[s_{j_l\theta}\circ s_{j_{l-1}\theta}\circ \cdots \circ s_{j_{1}\theta}] \ \lambda_{Y^\theta} \ 
[s_{j_l\theta}\circ   a'_{j_{l-1}\theta}\circ \dots\circ a'_{j_1\theta}] \ \lambda_{Y^\theta} \
[a'_{j_l\theta}\circ a'_{j_{l-1}\theta}\circ \cdots \circ a'_{j_{1}\theta}]\]
in $\mathscr{GP}$.  Finite induction finishes the proof.
\end{proof}

Again, for any reduction function $\theta=\theta(h,k,I,J)$ let 
\[T^{\theta}=\{ ([p\circ a'_{j_k\theta}\circ \dots\circ a'_{j_1\theta}],[q\circ a'_{j_k\theta}\circ \dots\circ a'_{j_1\theta}]): (p,q)\in X_{i_\ell\theta}, 1\leq l\leq h\}.\]
The proof of the following is similar to that of Lemma~\ref{lem:craigtrick}.

\begin{Lem}\label{lem:redfuneasy}\label{lem:red2} With notation as above, we have
\[\{ ([p\circ a'_{j_k\theta}\circ \dots\circ a'_{j_1\theta}],[q\circ a'_{j_k\theta}\circ \dots\circ a'_{j_1\theta}]): (p,q)\in \mathbf{l}(a_{i_{\ell\theta}}), 1\leq l\leq h\} \subseteq \mathbf{l}([a]).\]
Consequently, $T^\theta \subseteq \mathbf{l}([a])$. Further, if $(s_{i_l\theta}, s'_{i_l\theta})\in \mathbf{l}(a_{i_l\theta})$ for $1\leq l\leq h$, then 
\[ [s_{i_l\theta} \circ a'_{j_k\theta}\circ \dots\circ a'_{j_1\theta}] \,  \lambda_{T^\theta}\,[s'_{i_l\theta}\circ a'_{j_k\theta}\circ \dots\circ a'_{j_1\theta}].\]
In particular, if $s_{i_l\theta}a_{i_l\theta}=a_{i_l\theta}$ then 
\[ [s_{i_l\theta} \circ a'_{j_k\theta}\circ \dots\circ a'_{j_1\theta}] \,  \lambda_{T^\theta}\,[a'_{j_k\theta}\circ \dots\circ a'_{j_1\theta}].\]\end{Lem}

We now let \[
Y=\bigcup_{\theta\in \mathscr{R}}Y^{\theta},\, T=\bigcup_{\theta\in \mathscr{R}}T^{\theta}\mbox{ and }P=Y\cup T\]
where $\mathscr{R}=\mathscr{R}(a)$ is the set of all reduction functions for $a$. From Lemmas~\ref{lem:craigtrick} and ~\ref{lem:redfuneasy} we have that
$Y\cup T\subseteq \mathbf{l}([a])$. Note that $P$ is finite.
We will add later a further finite set of pairs from $\mathbf{l}([a])$ to $Y\cup T$, which will depend on parallel reduction functions, and show that the resulting finite set generates $\mathbf{l}([a])$.

We now consider  a fixed pair  $([s],[t])\in \mathbf{l}([a])$, where we assume $s$ and $t$ are reduced. Proposition~\ref{prop:reduction} sets out very precisely how to
reduce the products $s\circ a$ and  $t\circ a$, using reduction functions. The pair of reduction functions used will give us a new generator for $\mathbf{l}([a])$.

We suppose that $s\circ a$ reduces to a reduced word in $M$ reduction moves. From Proposition~\ref{prop:reduction}
we have a reduction function $\theta=\theta(h,k,I,J)$ for $a$ and words $ s^M$ and $a^M$ such that 
$[1,M]=I \cup J$ where $I$ labels the glueing moves and $J$ labels the deletion moves.
We write $I=\{ i_\ell: 1\leq \ell\leq h\}$ where  $1\leq i_1<\dots <i_h\leq M$,  and
$J=\{ j_o: 1\leq o\leq k\}$ where $1\leq j_1<\dots <j_k\leq M$; notice that $M=h+k$ and
$M$ is the greater of $i_h$ and $j_k$.
Further,  $\theta:[1,M]\rightarrow [1,m]$ is an injection such that:
\begin{enumerate}
\item[$\bullet$] $s$ shuffles to $s^M\circ  s_{i_1\theta}\circ \dots \circ s_{i_h\theta} \circ s_{j_k\theta}\circ \dots\circ s_{j_1\theta}$ where $s_{j_o\theta}a_{j_o\theta}=1$ for all $1\leq o\leq k$
and  $s^M$ is the word $s$ with the letters  $s_{ i\theta}$ deleted, $1\leq i\leq M$;
\item[$\bullet$]  $a$ shuffles to $a_{j_1\theta}\circ\cdots\circ a_{j_k\theta}\circ a_{i_h\theta}\circ\cdots\circ a_{i_1\theta}\circ a^M$,
where   $a^M$ is the word $a$ with the letters  $a_{ i\theta}$ deleted, $1\leq i\leq M$;
\item[$\bullet$] $s\circ a\equiv s^M\circ s_{i_1\theta}a_{i_1\theta}\circ\dots\circ s_{i_h\theta}a_{i_h\theta}\circ  a^M$ and the right hand side is reduced;
\item[$\bullet$] $s_{i_1\theta}a_{i_1\theta}\circ\dots\circ s_{i_h\theta}a_{i_h\theta} $ is a complete block.
\end{enumerate}

We suppose that $t\circ a$ reduces to a reduced word in $N$ reduction moves. From Proposition~\ref{prop:reduction}
we have a reduction function $\psi=\psi(p,q,U,V)$ for $a$ and words $t^N$ and $a^N$ such that 
$[1,N]=U \cup V$ where $U$ labels the glueing moves and $V$ labels the deletion moves.
We write $U=\{ u_\ell: 1\leq \ell\leq p\}$ where  $1\leq u_1<\dots <u_p\leq N$,  and
$V=\{ v_o: 1\leq o\leq q\}$ where $1\leq v_1<\dots <v_q\leq N$; notice that  $N=p+q$ and
$N$ is the greater of $u_p$ and $v_q$.
Further,  $\psi:[1,M]\rightarrow [1,m]$ is an injection such that:
\begin{enumerate}
\item[$\bullet$] $t$ shuffles to $t^N\circ  t_{u_1\psi}\circ \dots \circ t_{u_p\psi} \circ t_{v_q\psi}\circ \dots\circ t_{v_1\psi}$ where $t_{u_o\psi}a_{u_o\psi}=1$ for all $1\leq o\leq q$
and  $t^N$ is the word $t$ with the letters  $t_{ i\psi}$ deleted, $1\leq i\leq N$;
\item[$\bullet$]  $a$ shuffles to $a_{u_1\psi}\circ\cdots\circ a_{u_p\psi}\circ a_{v_q\psi}\circ\cdots\circ a_{v_1\psi}\circ a^N$,
where   $a^N$ is the word $a$ with the letters  $a_{ i\psi}$ deleted, $1\leq i\leq N$;
\item[$\bullet$] $t\circ a\equiv t^N\circ t_{u_1\psi}a_{u_1\psi}\circ\dots\circ t_{v_q\psi}a_{v_q\psi}\circ  a^N$ and the right hand side is reduced;
\item[$\bullet$] $t_{u_1\psi}a_{u_1\psi}\circ\dots\circ t_{u_p\psi}a_{u_p\psi} $ is a complete block.
\end{enumerate}

The reader should note here that neither $a^M$ nor $a^N$ need be a subword of the other, although both are subwords of $a$.

We call the pair of reduction functions  $(\theta,\psi)$ a {\em parallel pair} and say that $([s],[t])\in \mathbf{l}([a])$ {\em belongs to} a  {\em  $(\theta,\psi)$ reduction}. Note that there are only finitely many types of parallel pairs in  $\mathbf{l}([a])$, since there
are only finitely many reduction functions for $a$.

Let  \[A=s_{i_1\theta}a_{i_1\theta}\circ\dots\circ s_{i_h\theta}a_{i_h\theta}\circ  a^M
 \mbox{  and }B=t_{u_1\psi}a_{u_1\psi}\circ\dots\circ t_{u_p\psi}a_{u_p\psi}\circ  a^N,\]
 so that \[s^M\circ A\equiv t^N\circ B.\]
From Proposition~\ref{prop:factorisation}  we have that there are subwords $A'$ of $A$ and $B'$ of $B$ and a word $w\in X^*$ such that
\[ s^M\equiv w\circ B', \, B'\circ A \equiv A'\circ B\mbox{ and } t^N\equiv w\circ A'\]
where both $w\circ B'\circ A$ and $w\circ A'\circ B$ are reduced.

From the first and last equivalences above, we obtain that \[([s],[t])=([w\circ \bar{s}],[w\circ \bar{t}])=([w][ \bar{s}],[w][ \bar{t}]),\] where
\[\bar{s}=B'\circ s_{i_1\theta}\circ \dots \circ s_{i_h\theta} \circ s_{j_k\theta}\circ \dots\circ s_{j_1\theta}\]
and
\[ \bar{t}=A' \circ t_{u_1\psi}\circ \dots \circ t_{u_p\psi} \circ t_{v_q\psi}\circ \dots\circ t_{v_1\psi}.\]
Notice that $\bar{s}\circ a\equiv B'\circ A$ and $\bar{t}\circ a\equiv A'\circ B$, and so the second  equivalence yields
\[([\bar{s}],[\bar{t}])\in\mathbf{l}([a]).\]

 Now consider $\bar{s}$ and $\bar{t}$. We have so far
\[\bar{s}\circ a\equiv B'\circ  A\equiv A'\circ B\equiv \bar{t}\circ a\]
where $B'$ is a subword of $B$ and $A'$ is a subword of $A$. It is convenient to define
\[B'=B'_1\circ B'_2\mbox{ and }A'=A'_1\circ A'_2\]
where  $B'_1$ is a subword of $t_{u_1\psi}a_{u_1\psi}\circ \cdots \circ t_{u_p\psi}a_{v_p\psi}$ and
$B_2'$ is a subword of $a^N$ and where similarly $A'_1$ is a subword of $s_{i_1\theta}a_{i_1\theta} \circ \dots\circ s_{i_h\theta}a_{i_h\theta}$ and
$A_2'$ is a subword of $a^M$.

\begin{Lem}\label{lem:mathch} Suppose that the letter $s_{i_l\theta}a_{i_l\theta}$ of $A$ shuffles to some $t_{u_j\psi}a_{u_j\psi}$ of $B$ in the shuffle
from $B'\circ A$ to $A'\circ B$. Then $i_l\theta=u_j\psi$ and $s_{i_l\theta}a_{i_l\theta}=t_{u_j\psi}a_{u_j\psi}$ shuffles to the front of $A'\circ B$ and $B'\circ A$.
\end{Lem}
\begin{proof} Let $\alpha=\supp(a_{i_l\theta})=\supp(a_{u_j\psi})$. In the reduction of $s\circ a$, by  virtue of $a_{i_l\theta}$ not having been deleted, neither has any element of $M_{\alpha}$ occurring to the right of $a_{i_l\theta}$, that is,
any $a_o\in M_{\alpha}$ with ${i_l\theta}<o$. Similarly, in  the reduction of $t\circ a$, by  virtue of $a_{u_j\psi}$ not having been deleted, neither has any element of $M_{\alpha}$ occurring to the right of $a_{u_j\psi}$, that is,
any $a_o\in M_{\alpha}$ with ${u_j\psi}<o$. It follows that the number $r_A$ of elements of
$M_{\alpha}$ to the right of $s_{i_l\theta}a_{i_l\theta}$ in $B'\circ A$ and hence in $A$ is the same as the number of elements of $M_{\alpha}$
to the right of $a_{i_l\theta}$ in $a$. Similarly,  the number $r_B$ of elements of
$M_{\alpha}$ to the right of $t_{u_j\psi}a_{u_j\psi}$ in $A'\circ B$ and hence in $B$ is the same as the number of elements of $M_{\alpha}$
to the right of $a_{u_j\psi}$ in $a$.
Our assumption is that in the shuffle from $B'\circ A$ to $A'\circ B$, we have that $s_{i_l\theta}a_{i_l\theta}$ shuffles to  $t_{u_j\psi}a_{u_j\psi}$. It follows from the above and the remark before Proposition \ref{easy observation} that $r_A=r_B$ and consequently,
$i_l\theta=u_j\psi$.
The remaining statement follows from Lemma~\ref{lem:thedoubleshuffle}.
\end{proof}

\begin{Lem}\label{lem:setto1} If in  the shuffle from $B'\circ A$ to $A'\circ B$ we have that
 $s_{i_l\theta}a_{i_l\theta}$ shuffles to a letter $a_j$ of $a^N$, then $j=i_l\theta$. In particular, $i_l\theta\notin \{ v_1\psi, \cdots, v_q\psi\}$.

 Suppose that the letter  $s_{i_l\theta}a_{i_l\theta}$ of $B'\circ A$ does not shuffle to the front.  Then in the shuffle $B'\circ A\equiv A'\circ B$ we have that
 $s_{i_l\theta}a_{i_l\theta}$ shuffles to the letter $a_{i_l\theta}$ of $a^N$. In particular, $i_l\theta\notin \{ v_1\psi, \cdots, v_q\psi\}$.
\end{Lem}
\begin{proof} The same argument as  in Lemma \ref{lem:mathch} gives the first part.

 Suppose that the letter  $s_{i_l\theta}a_{i_l\theta}$ of $B'\circ A$ does not shuffle to the front. Therefore, in the shuffle from $B'\circ A$ to $A'\circ B$, we cannot have that $s_{i_l\theta}a_{i_l\theta}$ shuffles to $A'$. If $s_{i_l\theta}a_{i_l\theta}$
shuffles to some $t_{u_j\psi}a_{u_j\psi}$ then by  Lemma~\ref{lem:mathch} it would have to shuffle to the front. We conclude that $s_{i_l\theta}a_{i_l\theta}$ shuffles to some letter $a_{i_l\theta}$ of $a^N$.
\end{proof}

We follow the process outlined in Lemma~\ref{lem:thedoubleshuffle}, with $A$ in place of $a$ and $B$ in place of $b$, and with $A^{\ell}= s_{i_1\theta}a_{i_1\theta}\circ \dots \circ s_{i_h\theta}a_{i_h\theta}$ and $B^{\ell}= t_{u_1\psi}a_{u_1\psi}\circ \dots \circ t_{u_p\psi}a_{u_p\psi}$.
Let
\[s_{i_1\theta}\circ \dots \circ s_{i_h\theta} =s_1\circ s_2\mbox{ and }
t_{u_1\psi}\circ \dots \circ t_{u_p\psi} =t_1\circ t_2 \] where
$s_1$ is the product of all $s_{i_l\theta}$ such that  ${i_l\theta}\in K_A$, that is, $s_{i_l\theta}a_{i_l\theta}$ shuffles to some $a_j$ in $a^N$,  (and hence by Lemma~\ref{lem:setto1} to  $a_{i_l\theta}$) in the shuffle from $B'\circ A$ to $A'\circ B$;  and let $s_2$ denote the product of the remaining letters in
$s_{i_1\theta}\circ \dots \circ s_{i_h\theta} $; the words $t_1$ and $t_2$ are defined dually.  From Lemma~\ref{lem:setto1} we have that  for any $s_{i_l\theta}$ in $s_2$ we have that  $s_{i_l\theta}a_{i_l\theta}$ shuffles to the front of $B'\circ A$ and dually, for any $t_{u_k\psi}$ in $t_2$, we have that
$t_{u_k\psi}a_{u_k\psi}$ shuffles to the front of $A'\circ B$. Considerations of support immediately give the following.

\begin{Cor}\label{cor:s_2tofront} We have  $B'\circ s_2\equiv s_2\circ B'$ and   dually,  $A'\circ t_2\equiv t_2\circ A'$. 
\end{Cor}

The following comes from repeated applications of Lemma~\ref{lem:redfuneasy}. 

\begin{Lem}\label{lem:red3}  With definitions as above we have
\[ [s_1  \circ a'_{j_k\theta}\circ \dots\circ a'_{j_1\theta}]\,\lambda_{T}\,
[a'_{j_k\theta}\circ \dots\circ a'_{j_1\theta}].\]
and
\[ [t_1  \circ a'_{v_q\psi}\circ \dots\circ a'_{v_1\psi}]\, \lambda_{T} \,
[a'_{v_q\psi}\circ \dots\circ a'_{v_1\psi}].\]
\end{Lem}

\begin{defn}\label{defn:tricky} For a parallel pair of reduction functions $(\theta,\psi)$  and $B'$ and $A'$ defined as above, we let
$B'_r$ be $B'$ where we replace any $t_{u_k\psi}a_{u_k\psi}$ in $B_1$ by $a_{u_k\psi}$ (so that $B_r'$ is a product of letters of $a$)
and let $A'_r$ be $A'$ where  we replace any $s_{i_l\theta}a_{i_l\theta}$ in $B_1$ by $a_{i_l\theta}$ (so that $A_r'$ is a product of letters of $a$). We then set
\[z^{\theta,\psi}_\theta=B'_r\circ a'_{j_k\theta}\circ \dots\circ a'_{j_1\theta}\mbox{ and }z^{\theta,\psi}_\psi=A'_r\circ a'_{v_q\psi}\circ \dots\circ a'_{v_1\psi}.\]
\end{defn}

\begin{Lem}\label{lem:zz'} We have that $ ([z^{\theta,\psi}_{\theta}],[z^{\theta,\psi}_{\psi}])\in \mathbf{l}([a])$.
\end{Lem}
\begin{proof} We have that
\[\begin{array}{rcl}
[z^{\theta,\psi}_{\theta}][ a]=[z^{\theta,\psi}_{\theta}\circ a]&=&[ B_r'\circ a'_{j_k\theta}\circ \dots\circ a'_{j_1\theta} \circ a_{j_1\theta}\circ \dots\circ a_{j_k\theta}\circ a_{i_1\theta}\circ\cdots\circ  a_{i_h\theta}\circ a^M]\\
&=&[B_r'\circ a_{i_1\theta}\circ\cdots\circ  a_{i_h\theta}\circ a^M]
\end{array}\]
and similarly
\[[z^{\theta,\psi}_{\psi}][a]=[ A_r' \circ a_{u_1\psi}\circ\cdots\circ  a_{u_p\psi}\circ a^N].\]
Let $A_r$ be obtained from $A$ by replacing any $s_{j_l\theta}a_{j_l\theta}$ by $a_{j_l\theta}$ and let
$B_r$ be obtained from $B$ by replacing any $t_{u_l\psi}a_{u_l\psi}$ by $a_{u_l\psi}$, so that 
$[z^{\theta,\psi}_{\theta}\circ a]=[B_r'\circ A_r]$ and $[z^{\theta,\psi}_{\psi}\circ a]=[A_r'\circ B_r]$. 
We know that $A'\circ B\equiv B'\circ A$. ~An application of Lemma~\ref{lem:reduction-3} now gives that
$A'_r\circ B_r\equiv B'_r\circ A_r$. Consequently, $[(z^{\theta,\psi}_{\theta}],[z^{\theta,\psi}_{\psi}])\in \mathbf{l}([a])$, as claimed. 
\end{proof}

We now let
\[K=P\cup \{ ([z^{\theta,\psi}_{\theta}],[z^{\theta,\psi}_{\psi}]): (\theta,\psi)\mbox{ is a parallel reduction pair}\}.\]

We return to consideration of  $\bar{s}$. We have 
\[\begin{array}{rcl}
[\bar{s}]&=& [B'\circ s_2\circ s_1\circ   s_{j_k\theta}\circ \dots\circ s_{j_1\theta}]\\
&\lambda_K&   [B'\circ s_2 \circ s_1 \circ a'_{j_k\theta}\circ \dots\circ a'_{j_1\theta}]\\
&\lambda_K & [B'\circ s_2 \circ a'_{j_k\theta}\circ \dots\circ a'_{j_1\theta}]\\
&=& [ B_1'\circ s_2\circ B_2'\circ a'_{j_k\theta}\circ \dots\circ a'_{j_1\theta}]\end{array} \]
where notice that $B_2'$ just consists of fixed letters of $a$.

Similarly,

\[[\bar{t}]\,  \lambda_K\,  [A'\circ t_2 \circ a'_{v_q\psi}\circ \dots\circ a'_{v_1\psi}]= [A_1'\circ t_2\circ A_2'\circ a'_{v_q\psi}\circ \dots\circ a'_{v_1\psi}] \]
where $A_2'$ just consists of fixed letters of $a$.

We now return to examine $s_2$ and $t_2$. We have that $s_2$ is the product of the letters $s_{i_l\theta}$ in
$s_{i_1\theta}\circ \dots \circ s_{i_h\theta} $ such that $s_{i_l\theta}a_{i_l\theta} $ does not shuffle to $a^N$ in the shuffle $B'\circ A$ to $A'\circ B$.  Let $a_2$  be the subword of $a$ with the same support as $s_2$. We further divide $s_2$ as a product $s_3\circ s_4$. Here  $s_3$ is the product of all the letters $s_{i_l\theta}$  such that $s_{i_l\theta}a_{i_l\theta} $ shuffles to some $t_{u_j\psi}a_{u_j\psi}$, that is, $i_l\theta\in I_A$ (recall that Lemma~\ref{lem:mathch} gives in this case $i_l\theta=u_j\psi$). Correspondingly,  $s_4$ denotes the remaining letters, that is, those $s_{i_l\theta}$ such that 
${i_l\theta}\in J_A$, so that $s_{i_l\theta}a_{i_l\theta} $ shuffles to a letter in $A'$. We define $a_3$ and $a_4$ in the corresponding way. Then 
$s_2\circ a_2$, $s_3\circ a_3$ and $s_4\circ a_4\equiv A_1'$ are the corresponding product of the letters in  $s_{i_1\theta}a_{i_1\theta}\circ \dots \circ s_{i_h\theta}a_{i_h\theta} $ (by virtue of the completeness of the underlying block). The words $t_2,t_3$ and $t_4$ and $a^2, a^3=a_3$ and $a^4$ are defined dually.

The above considerations give us that $s_3a_3\equiv t_3a_3$. For each $i_{l\theta}\in\supp(s_3)=\supp(t_3)$  
have $s_{i_l\theta}a_{i_l\theta}=t_{i_l\theta}a_{i_l\theta}$ and so $(s_{i_l\theta},t_{i_l\theta})\in\mathbf{l}(a_{i_l\theta})$.
Repeated applications of Lemma~\ref{lem:redfuneasy} yield  that 
\[[s_3\circ a'_{j_k\theta}\circ \dots\circ a'_{j_1\theta}]\, \lambda_K\, [t_3\circ a'_{j_k\theta}\circ \dots\circ a'_{j_1\theta}]\]
and similarly
\[[s_3\circ a'_{v_q\psi}\circ \dots\circ a'_{v_1\psi}]\, \lambda_K\, [t_3\circ a'_{v_q\psi}\circ \dots\circ a'_{v_1\psi}]\]

\begin{Lem}\label{lem:finally}  We have that 
$[s]\,\lambda_K\, [t].$
\end{Lem}
{\begin{proof} We know that $([s],[t])=([w][\bar{s}],[w]\bar{t}])$, so it is enough to show that $[\bar{s}]\,\lambda_K\, [\bar{t}].$
We have
\[\begin{array} {rcl}
[\bar{s}]&\lambda_K&[B'\circ s_2\circ a'_{j_k\theta}\circ \dots\circ a'_{j_1\theta}]\\
&=&[B'\circ s_4\circ s_3\circ a'_{j_k\theta}\circ \dots\circ a'_{j_1\theta}]\\
&=& [B'\circ s_4][s_3\circ a'_{j_k\theta}\circ \dots\circ a'_{j_1\theta}]\\
&\lambda_K& [B'\circ s_4][t_3\circ a'_{j_k\theta}\circ \dots\circ a'_{j_1\theta}]\\
&=& [B'\circ s_4\circ t_3\circ a'_{j_k\theta}\circ \dots\circ a'_{j_1\theta}]\\
&=& [s_4\circ t_3\circ B'\circ a'_{j_k\theta}\circ \dots\circ a'_{j_1\theta}]\\
&=& [s_4\circ t_3\circ B_1'\circ B_2'\circ a'_{j_k\theta}\circ \dots\circ a'_{j_1\theta}]\\
&=& [s_4\circ t_3\circ t_4a^4 \circ B_2'\circ a'_{j_k\theta}\circ \dots\circ a'_{j_1\theta}]\\
&=& [s_4\circ t_3\circ t_4\circ a^4 \circ B_2'\circ a'_{j_k\theta}\circ \dots\circ a'_{j_1\theta}]\\
&=& [s_4\circ t_3\circ t_4\circ B'_r\circ a'_{j_k\theta}\circ \dots\circ a'_{j_1\theta}]\\
&=& [s_4\circ t_3\circ t_4\circ z^{\theta,\psi}_{\theta}]\\
&\lambda_K& [s_4\circ t_3\circ t_4\circ z^{\theta,\psi}_{\psi}]\\
\end{array}\]using Corollary~\ref{cor:s_2tofront} and the fact that $s_3$ and $t_3$ have the same support.
By not changing $s_3$ to $t_3$ at line 4 in the above display, we also have
\[[\bar{s}]\,\lambda_K\, [s_4\circ s_3\circ t_4\circ z^{\theta,\psi}_{\psi}]\]
and so by duality, 
\[[\bar{t}]\,\lambda_K\, [s_4\circ t_3\circ t_4\circ z^{\theta,\psi}_{\theta}]\] finally yielding 
$[\bar{s}]\,\lambda_K\,[\bar{t}]$, and so $[{s}]\,\lambda_K\,[{t}]$, as required.
\end{proof}} 
\end{proof}

The above finishes the proof of Theorem~\ref{thm:aagh}. We may immediately deduce the corresponding result for free products and restricted direct products.

\begin{Cor} cf. \cite{dasar:2024}\label{cor:aagh}  A free product or restricted direct product of monoids $\{ M_i:i\in I\}$ is finitely left equated if and only if so is each $M_i,i\in I$.
\end{Cor}

\section{Weak left coherency} \label{sec:wlc}

A monoid $M$ is said to be {\it weakly left  coherent} if every finitely generated left
 ideal of $M$ is finitely presented as a left $M$-act.  It was shown in  \cite[Corollary 3.3]{gould:1992} that a monoid $M$ is weakly left coherent if and only if $M$ is both left ideal Howson and finitely left equated. As we mentioned earlier, for rings the notions of being weakly left  coherent and left coherent coincide. However,  left coherency for monoids is strictly
 stronger than weak left coherency. For instance, infinite full transformation monoids are
 weakly left coherent \cite[Corollary 3.6, Theorem 3.7]{gould:1992}  but not left coherent \cite[Theorem 3.4]{brookes:2025}. Further, free inverse monoids of rank $\geq 2$ are weakly left coherent but not
left coherent  \cite[ Theorem 7.5]{gould:2017}.

Combining Theorems \ref{thm:howson} and \ref{thm:aagh}, we obtain:

\begin{Thm}\label{thm:wlc} A graph product of monoids $\mathscr{GP}=\mathscr{GP}(\Gamma,\mathcal{M})$ is weakly left coherent if and only if so is each vertex monoid.
\end{Thm}
\begin{Cor}\label{cor:wlc} \cite{dasar:2024} A free product or restricted direct product of monoids $\{ M_i:i\in I\}$ is weakly left coherent if and only if so is each $M_i,i\in I$.
\end{Cor}

We have concentrated in this article on finitary conditions arising from and related to left  noetherianity, the  {\em ascending} chain condition on the lattice of left congruences of a monoid. There are, of course, conditions that arise from the {\em descending} chain condition on the lattice of left congruences or, indeed, chain conditions on the lattice of (two-sided) congruences. We suggest that an investigation of these conditions in the context of direct products, and then potentially that of  graph and free products, would be worthwhile. It is known that the direct product of two Artinian groups (that is, groups with no infinite strictly  descending chain of subgroups) is Artinian \cite{robinson:1995}. The corresponding result is not known for monoids, in the case of left ideals or of left congruences. It is clear, however, that a free product of two non-trivial monoids has an infinite descending chain of principal left ideals.

\section{Acknowledgements} We are grateful to the referee for a careful reading of the manuscript and to  Levent Dasar for pointing out an error in the original formulation of Proposition~\ref{key15}. The first author would like to thank the Department of Mathematics of the University of York for hosting her 
whilst this work was in development.

\end{document}